\documentclass[11pt]{article}

\usepackage[margin=1in]{geometry}
\usepackage{amsmath,amssymb,amsfonts,amsthm,mathtools,mathrsfs}
\usepackage{bm}
\usepackage{enumitem}
\usepackage{xcolor}
\usepackage{graphicx}
\usepackage{tikz}
\usetikzlibrary{arrows.meta}
\usepackage{natbib}
\usepackage[colorlinks=true,linkcolor=blue,citecolor=blue,urlcolor=blue]{hyperref}
\makeatletter
\renewcommand*{\hyper@natlinkbreak}[2]{#1}
\makeatother

\newtheorem{theorem}{Theorem}[section]
\newtheorem{proposition}[theorem]{Proposition}
\newtheorem{lemma}[theorem]{Lemma}
\newtheorem{corollary}[theorem]{Corollary}
\newtheorem{remark}[theorem]{Remark}

\newcommand{\R}{\mathbb{R}}
\newcommand{\E}{\mathbb{E}}
\newcommand{\1}{\mathbf{1}}
\newcommand{\dd}{\,\mathrm{d}}
\newcommand{\cD}{\mathbb{D}}

\DeclareMathOperator{\Pois}{Pois}
\DeclareMathOperator{\IG}{IG}

\title{A Hawkes Microfoundation for Multitype Inverse Gaussian Subordinators}
\author{
Yingli Wang\thanks{School of Mathematical Sciences, Fudan University, Shanghai, China; \texttt{yingliwang@fudan.edu.cn}}
\and
Wei Xu\thanks{School of Mathematics and Statistics, Beijing Institute of Technology, Beijing, China; \texttt{xuwei.math@gmail.com}}
\and
Lingjiong Zhu\thanks{Department of Mathematics, Florida State University, Florida, United States of America; \texttt{zhu@math.fsu.edu}}}
\date{\today}

\begin{document}

\maketitle

\begin{abstract}
We provide an event-level Hawkes microfoundation for a multitype inverse-Gaussian stochastic clock.  We show that the event counts and integrated intensities of nearly critical multivariate linear Hawkes processes converge jointly to a multivariate pure-jump subordinator when reproduction delays have finite mean and immigration is balanced so that rare large families remain visible.  The dependence among the limiting coordinates is inherited from microscopic cross-excitation, and the limit admits a Brownian additive-field first-passage representation.  We call this process the multitype inverse-Gaussian subordinator.  Its diagonal case consists of independent classical inverse-Gaussian subordinators, with the univariate model as a further special case.  The square-root specialization recovers the inverse-Gaussian clock obtained at the boundary of hyper-rough square-root models by Abi Jaber--Attal--Rosenbaum (\textit{Ann. Appl. Probab.} \textbf{36}(4): 3635--3660, 2026).  Our cluster proof exposes the underlying mechanism: finite-variance near-critical branching creates rare but macroscopic families, while their internal timing disappears on the observation scale, so each family becomes a jump.  We also clarify the relation with the multivariate Hawkes scaling theory of Xu (arXiv:2412.14459): its diagonal-atom condition yields the diagonal specialization, whereas the genuinely coupled limit requires a separate uniqueness argument.  In addition, under a common tilted-stability condition, we replace the high-intensity assumption in that theory by a weaker accumulated-activity condition and provide the Riccati identification needed when the potential has an atom.  We strengthen convergence of the count and compensator and extend the scalar conclusion to finite-variance age-dependent branching clusters.
\end{abstract}



\section{Introduction}

Let $N$ be a simple point process on $\R_+$ and let
\[
    \mathcal F_t 
    :=\sigma\!\left(N(C):C\in\mathcal B(\R),\ C\subset(0,t]\right)
\]
be an increasing family of $\sigma$-algebras.  Any nonnegative $(\mathcal F_t)$-progressively measurable process $\lambda_t$ satisfying
\[
    \E\!\left[N(a,b]\mid\mathcal F_a\right]
    =\E\!\left[
        \int_a^b\lambda_s\dd s
        \,\middle|\,\mathcal F_a 
      \right]
    \qquad\text{almost surely},
\]
for all intervals $(a,b]$ is called an $(\mathcal F_t)$-intensity of $N$.  We use the notation $N_t:=N(0,t]$ to denote the number of points in the interval $(0,t]$.

We consider $N$ to be a univariate linear Hawkes process admitting the $(\mathcal F_t)$-intensity
\begin{equation}\label{eq:hawkes-intro}
    \lambda_t
    =\mu+\int_{(0,t)}h(t-s)\dd N_s
    =\mu+\sum_{\tau_i<t}h(t-\tau_i), \quad t\geq 0,
\end{equation} 
where $\mu>0$ is the \textit{baseline intensity}, $\{\tau_i\}_{i\geq 1}$ are the event times, and $h:\R_+\to\R_+$ is a measurable, nonnegative \textit{excitation kernel} satisfying $h\in L^1(\R_+)$.  The value $h(u)$ describes the increment in the current intensity caused by an event that occurred $u$ time units earlier.  Thus every arrival raises the future intensity and generates the self-excitation and temporal clustering characteristic of the model.  Introduced by \cite{Hawkes1971}, the linear Hawkes process naturally generalizes the Poisson process and remains tractable for statistical analysis.  In finance, it has been used to model clustered trades and quote changes, high-frequency price movements, limit-order-book events, and financial contagion; see, for example, \cite{Bowsher2007,EmbrechtsLinigerLin2011} and the survey of \cite{BacryMastromatteoMuzy2015}.
The total kernel mass
\[
    a:=\|h\|_{L^1}=\int_0^\infty h(s)\dd s
\]
governs the stability and branching behavior of a linear Hawkes process.  
 The process is called \textit{subcritical}, \textit{critical} or \textit{supercritical} when $a<1$, $=1$ or $>1$ respectively.  
 By the immigration--branching representation of \cite[pp.~495--496, especially equations~(5)--(6)]{HawkesOakes1974}, $a$ is also the mean number of direct offspring produced by an event and is therefore called the branching ratio.  In the subcritical case, every immigrant family is finite almost surely, whereas at criticality the mean total family size diverges.  A triangular array of subcritical Hawkes processes with branching ratios $a_T\uparrow1$ is called nearly unstable, or near-critical.  As the branching ratio approaches one, very large families remain rare but become increasingly consequential.  Critical Hawkes systems, including constructions without ancestors, have consequently been studied through both point-process and branching methods \cite{BremaudMassoulie2001,Kirchner2023}.

For such arrays,  scaling limits of nearly unstable Hawke models scaling connect high-frequency event data to macroscopic stochastic volatilities.  
 In the finite-mean window $T(1-a_T)\to\lambda\in(0,\infty)$, under some light-tailed assumptions on the kernel, the normalized intensity converges to a \textit{Cox--Ingersoll--Ross model} (CIR-model) \cite{JaissonRosenbaum2015}. 
 On another hand, heavy-tailed reproduction times instead retain long memory and lead to fractional or rough limits \cite{JaissonRosenbaum2016}.  Related Hawkes microfoundations for leverage and rough volatility are developed in \cite{ElEuchFukasawaRosenbaum2018,JusselinRosenbaum2020}, with a multidimensional Hawkes-to-rough-volatility extension in \cite{RosenbaumTomas2021}.  Recent time-modulated heavy-tailed Hawkes limits are studied by \cite{GnabeyeuPagesRosenbaum2026}; see also \cite{HorstXuZhang2023}.  Complementary diffusion approximations for nearly critical multitype marked self-excited systems and Crump--Mode--Jagers branching processes are obtained in \cite{Xu2024Diffusion}; those limits retain population-density dynamics and yield multitype continuous-state branching diffusions, whereas our calendar-collapse regime retains total family masses as jumps.  
 Self-exciting jump cascades in stochastic-volatility microstructure are studied by \cite{HorstXu2022}, while the three growing-dimensional mean-field regimes are identified in \cite{SzymanskiXu2026}.  The long-run behavior of fixed subcritical and critical Hawkes processes is treated systematically in \cite{HorstXu2026}, while \cite{Xu2024Scaling} gives a broad multivariate triangular-array theory in terms of stochastic Volterra equations with measure kernels.  
 These results form part of the broader rough-volatility program initiated empirically in \cite{GatheralJaissonRosenbaum2018} and supported analytically by the affine Volterra framework \cite{AbiJaberLarssonPulido2019,AbiJaber2021}.

 A recent result of \cite{AbiJaberAttalRosenbaum2025} shows that when the Hurst parameter of a hyper-rough square-root process decreases to $-1/2$, a sequence of continuous increasing processes converges to an \textit{inverse Gaussian} (IG) pure-jump process.  
 The proof passes through a singular Volterra limit: a normalized fractional kernel collapses to a Dirac mass, a Volterra equation becomes a Brownian first-passage relation, and the resulting clock is IG \cite{AbiJaberAttalRosenbaum2025}.  
 This complements the fast-reversion route from rough volatility to normal inverse Gaussian models in \cite{AbiJaberDeCarvalho2024}.  The newer theory of Volterra clocks extends the same first-passage mechanism beyond affine boundaries \cite{AbiJaberAttalSojmark2026}, while inverse Gaussian variables also appear naturally in affine-Volterra simulation schemes for Hawkes processes \cite{AbiJaberAttalSotnikov2025}.  
 Under the notation in \cite{AbiJaberAttalRosenbaum2025}, their IG process parameters (equivalently, the unit-time marginal parameters) are reproduced, for $\nu\ne0$, by
\begin{equation}\label{eq:parameter-map}
    \eta=\frac{g_0}{|\nu|},
    \qquad
    \gamma=\frac{1+\lambda}{|\nu|}.
\end{equation}
The branching construction below therefore reaches the same Brownian first-passage object through a direct event-level limit.

 Related scalar limits are studied in \cite{Zitkovic2026Feller}, in which Feller random measures are represented as occupation measures of Dawson--Watanabe superprocesses with killed L\'evy-subordinator spatial motion.  
 For memory measure $\rho=\delta_0$ and immigration measure $\vartheta\dd t$, Example~5.10 of that work gives the scalar IG exponent $\vartheta(\sqrt{1+2u}-1)$.  This law is also contained in Theorems~2.9 and~2.12 in \cite{Xu2024Scaling} with $K=1$, $\Phi\equiv1$, $\Pi=\delta_0$, and $\Upsilon(t)=\vartheta t$.  Additionally, \cite{Zitkovic2026Feller} also proves the parameter stability for scalar Feller random measures and recovers the hyper-rough-boundary IG limit in Skorokhod space endowed with $M_1$-topology.

The present paper establishes a joint scaling limit for the counts and integrated intensities of mutually exciting Hawkes processes throughout the finite-mean calendar-collapse region.  The main theorem retains non-diagonal dependence among event types, and the cluster estimates identify the exact collapse condition for a fixed scalar reproduction-time density.  Our main theorem allows the kernel-mass matrix to approach the identity as
$\boldsymbol A_T=\boldsymbol I-r_T\boldsymbol B+o(r_T)$, where
$\boldsymbol B$ is a nonsingular $M$-matrix.  Its off-diagonal entries retain first-order cross-excitation in the limit.  For $\boldsymbol x=(x_1,\ldots,x_d)^{\top}$, write $\boldsymbol x^{\circ2}:=\left(x_1^2,\ldots,x_d^2\right)^{\top}$ for its componentwise square.  The limiting Laplace exponent is determined by the coupled quadratic system
\[
    \boldsymbol B^{\top}\boldsymbol z
       +\frac12\boldsymbol z^{\circ2}=\boldsymbol u,
\]
and admits a Brownian additive-field first-passage interpretation.  Diagonal $\boldsymbol B$ gives independent IG coordinates, and the classical scalar clock is recovered by $d=1$.

The scalar mechanism is especially transparent.  Write
\[
    r_T:=1-a_T.
\]
A finite-variance near-critical family has a total-progeny cutoff of order $r_T^{-2}$, reaches that scale with probability of order $r_T$, and has genealogical depth of order $r_T^{-1}$.  We therefore normalize counts by $c_T=r_T^{-2}$.  There are $T\mu_T$ immigrant families per macroscopic time unit, so the rare-family balance is
\begin{equation}\label{eq:scaling-intro}
    T\mu_Tr_T\longrightarrow\vartheta\in(0,\infty),\qquad\text{as $T\rightarrow\infty$}.
\end{equation}
Finite-mean reproduction times make the aggregate temporal profile disappear precisely under the additional condition $Tr_T\to\infty$.  An entire near-critical order-flow avalanche then survives only through its mass and becomes one jump.

 This produces a phase diagram rather than an isolated window.  
 If $r_T\asymp T^{-\beta}$, every $\beta\in(0,1)$ belongs to the IG calendar-collapse region after normalization by $r_T^{-2}$.  
 At the boundary $Tr_T\to\lambda\in(0,\infty)$, family depth remains macroscopic and one recovers the persistent-family, integrated-CIR regime of \cite{JaissonRosenbaum2015} under the corresponding immigration balance and the hypotheses of that theorem.  The case $Tr_T\to0$ is more persistent still and is not identified with that CIR boundary here.  The square-root choice is the regime in which, as $T\to\infty$,
\begin{equation}\label{eq:square-root-scaling}
    \sqrt T r_T\longrightarrow\gamma\in(0,\infty),
    \qquad
    \sqrt T\mu_T\longrightarrow\eta\in(0,\infty)
\end{equation}
is distinguished only if one insists that family mass be normalized by $T$: then $r_T^{-2}\asymp T$, and the limit has exactly the parameterization in \eqref{eq:parameter-map}.

The finite-mean assumption on reproduction times is essential to this comparison.  For heavy-tailed delays with index $\alpha\in(0,1)$, the duration of $n$ generations can be on the scale $n^{1/\alpha}$.  In particular, under the square-root specialization and $\alpha=1/2$, a family of depth $\sqrt{T}$ may still occupy order-$T$ calendar time, so its internal timing need not collapse and memory may survive in a Volterra limit \cite{JaissonRosenbaum2016}; see also \cite{HorstXuZhang2023}.  
The branching-ratio exponent alone therefore does not determine the macroscopic path type; the temporal tail of the Hawkes kernel matters as well.

 The square-root form of the IG exponent follows directly from the generating function of the total-progeny of a Poisson Galton--Watson family. 
 More precisely, let $S_T$ denote its total progeny, including the initial immigrant, when the offspring counts of different individuals are independent and identically distributed according to a Poisson law $\Pois(a_T)$, with rate $a_T=1-r_T<1$ and $a_T\to1$ as $T\to\infty$.  Set
\[
    q_T(u):=\E\left[e^{-ur_T^2S_T}\right].
\]
Conditioning on the initial individual's offspring count $\xi_T\sim\Pois(a_T)$ gives the total-progeny branching equation \cite{Minami2005}, here specialized to Poisson offspring:
\[
    q_T(u)=e^{-ur_T^2}\E\left[\left(q_T(u)\right)^{\xi_T}\right]
          =e^{-ur_T^2}\exp\left\{a_T\left(q_T(u)-1\right)\right\}.
\]
Writing $1-q_T(u)=r_Tz_T(u)$ yields
\[
    u=z_T(u)
       +\frac{-\log(1-r_Tz_T(u))-r_Tz_T(u)}{r_T^2},
\]
and hence,
\[
    z_T(u)\longrightarrow \sqrt{1+2u}-1.
\]
Multiplication by the asymptotic immigrant-family factor $T\mu_Tr_T\to\vartheta$ produces the IG exponent.

This calculation alone concerns a process in which every family is artificially collapsed at its immigrant time.  The main technical point is to return to the original Hawkes process.  We prove that the difference at any fixed time is bounded in $L^1$ by $\mu_Tm_\phi a_T$, using only the first moment of the reproduction-time density.  For the fixed scalar reproduction-time density, Lemma~\ref{lem:collapse-necessity} shows that the actual fixed-time error vanishes exactly when $Tr_T\to\infty$ under the rare-family balance.  Monotonicity, endpoint estimates, and finite-dimensional convergence then give process convergence in $M_1$.  The topology is essential: every normalized microscopic jump has size $r_T^2\to0$, whereas the limit has jumps.  The $J_1$ topology cannot merge a fast staircase of microscopic jumps into one macroscopic jump, while $M_1$ can.

 This topological issue also reveals that high microscopic intensity is not the intrinsic scaling requirement.  In the earlier work \cite{Xu2024Scaling}, one of us proved convergence of asymptotically critical multivariate Hawkes systems in Jakubowski's $S$-topology under the standing assumption $\theta_n\to\infty$.  
 Writing $b_n:=\sqrt{n\theta_n}$ shows that the rescaled count, compensator, martingale, resolvent, and criticality condition depend on $\theta_n$ only through $b_n$.  
 Under a common tilted-stability condition and the diagonal-atom assumption, we first prove these scaling conclusions under the weaker condition $n\theta_n\to\infty$.  
 Then, the positivity of the limiting potential and local $L^p$ estimates are established. 
 When the potential has an atom at zero, we also provide a fixed logarithm branch for the Riccati--Volterra existence and identification argument.  
 Finally, we prove a general monotone-path principle that upgrades the count and compensator to the product $M_1$ topology whenever their common limit is stochastically continuous.  
 This topology is the relevant one when the measure kernel possesses an atom and the limit may exhibit a jump; the $J_1$-topology does not handle this situation.
 We refere to \cite{Kern2023} for an expository comparison of the Skorokhod topologies. 

For the present normalization, matching
$\boldsymbol N^T_{Tt}/(T\theta_T)$ with
$r_T^2\boldsymbol N^T_{Tt}$ gives
$\theta_T=(Tr_T^2)^{-1}$ and hence
$T\theta_T=r_T^{-2}\to\infty$, regardless of the behavior of
$Tr_T^2$.  In the terminology of \cite{Xu2024Scaling}, our matrix scaling has
$\boldsymbol K=\boldsymbol I$, $\boldsymbol\Phi=\boldsymbol B$, and formal potential atom $\boldsymbol\Pi=\boldsymbol B^{-1}\delta_0$.  The diagonal-atom assumption in \cite{Xu2024Scaling} is satisfied in the diagonal specialization, including the scalar case, while its original high-intensity condition further requires $Tr_T^2\to0$.  Our accumulated-activity extension covers the diagonal specialization throughout the calendar-collapse region.  A non-diagonal $\boldsymbol B$ gives a non-diagonal potential atom; the multitype cluster proof identifies this coupled limit and the $M$-matrix argument gives uniqueness of its quadratic exponent.

Our main contributions are as follows.
\begin{enumerate}[label=(\roman*)]
    \item We introduce the multitype inverse Gaussian (IG) subordinator through the coupled quadratic system \eqref{eq:multitype-IG-root} and establish its representation as the first-passage time of a Brownian additive field. This construction recovers the univariate, diagonal multivariate, and square-root IG limits as successive special cases (Corollaries~\ref{cor:diagonal-IG}--\ref{cor:square-root}).
    
    \item We prove the multivariate Hawkes-to-IG limit throughout the finite-mean calendar-collapse region.  Coupled Poisson family transforms identify the limiting quadratic exponent, collapsed multitype families converge as compound-Poisson subordinators, and a matrix renewal estimate restores calendar time and yields joint product-$M_1$ convergence of the counts and compensators (Theorem~\ref{thm:main} and Lemmas~\ref{lem:multitype-family}--\ref{lem:multitype-collapse}).
    \item Under the common tilted stability and the diagonal-atom assumption, we show that the high-intensity condition $\theta_n\to\infty$ in the general multivariate Hawkes scaling theory can be replaced by the weaker accumulated-activity condition $n\theta_n\to\infty$. 
    Furthermore, we strengthen the weak convergence of the monotone count and its compensator to convergence in the product $M_1$-topology (Theorems~\ref{thm:relaxed-Xu} and~\ref{thm:monotone-M1}).  We also show that the diagonal-atom condition in \cite{Xu2024Scaling} recovers precisely the diagonal specialization of our limit.
    \item For a fixed scalar reproduction-time density, we prove that
    calendar-collapse at any fixed positive time occurs exactly when
    $Tr_T\to\infty$ under the rare-family balance
    (Lemma~\ref{lem:collapse-necessity}).
    
    \item We identify the limiting L\'evy measure from the near-critical Borel mass-function asymptotic, and give a Brownian first-passage interpretation of the immigrant forest.
    We also extend the IG limit to finite-variance age-dependent branching clusters (Propositions~\ref{prop:forest-passage} and~\ref{prop:universality}).
\end{enumerate}

 For background on L\'evy subordinators and IG/NIG families, we refer to \cite{Sato1999,Applebaum2009,ContTankov2004}.  The total-progeny identities used below go back to \cite{Otter1949,Dwass1969}, and branching first-passage limits are treated by \cite{PardouxWakolbinger2011}.  We include these ingredients to identify constants and expose the mechanism, not as separate novelty claims. 
 In recent work \cite{Talarczyk2025,Talarczyk2026}, stable limits for critical marked Hawkes systems further demonstrate how branching tails and Skorokhod topology combine to create jump limits.  
 Those heavy-mark limits are complementary to the finite-variance near-critical mechanism studied here.

 This paper is organized as follows.  Section~\ref{sec:model} states the multivariate model, defines the multitype IG limit, and records its diagonal and scalar specializations.  Section~\ref{sec:general} sharpens the general Hawkes scaling theory and explains exactly which specialization is covered by the diagonal-atom assumption in \cite{Xu2024Scaling}.  Section~\ref{sec:multiproof} proves the multivariate theorem by cluster methods.  Section~\ref{sec:collapsed} records the scalar family laws and proves the exact calendar-collapse condition.  Section~\ref{sec:branching} identifies the Borel tail, gives the Brownian forest interpretation, and establishes finite-variance universality, and Section~\ref{sec:conclusion} concludes.

\section{Multivariate model and main results}\label{sec:model}

Fix $d\ge1$.  For every $T\ge1$, let
$\boldsymbol\phi_T\in L^1(\R_+;\R_+^{d\times d})$ be a nonnegative matrix kernel and set
\begin{equation}\label{eq:matrix-moments}
    \boldsymbol A_T:=\int_0^\infty\boldsymbol\phi_T(s)\dd s,
    \qquad
    \boldsymbol D_T:=\int_0^\infty s\boldsymbol\phi_T(s)\dd s.
\end{equation}
The entry $\phi_{T,ij}(s)$ is the rate at which a type-$j$ event produces type-$i$ children after delay $s$.  Let $\boldsymbol N^T$ be the $d$-variate linear Hawkes process, started with empty history, with constant baseline $\boldsymbol\mu_T\in(0,\infty)^d$ and predictable intensity
\begin{equation}\label{eq:multivariate-hawkes}
    \boldsymbol\lambda_t^T
    =\boldsymbol\mu_T
      +\int_{(0,t)}\boldsymbol\phi_T(t-s)\dd\boldsymbol N_s^T,
    \qquad t\ge0.
\end{equation}
Its integrated intensity and compensated martingale are
\begin{equation}\label{eq:multivariate-compensator}
    \boldsymbol\Lambda_t^T:=\int_0^t\boldsymbol\lambda_s^T\dd s,
    \qquad
    \boldsymbol M_t^T:=\boldsymbol N_t^T-\boldsymbol\Lambda_t^T.
\end{equation}

Recall that a real matrix $\boldsymbol B$ is a nonsingular $M$-matrix if
$\boldsymbol B=s\boldsymbol I-\boldsymbol C$ for some nonnegative matrix
$\boldsymbol C$ and some $s>\rho(\boldsymbol C)$.  Equivalently,
$B_{ij}\le0$ for $i\ne j$ and $\boldsymbol B^{-1}$ is nonnegative.  We impose the following multivariate calendar-collapse assumptions for a scalar sequence $r_T>0$:
\begin{equation}\label{eq:multivariate-scaling}
\begin{gathered}
    r_T\to 0,
    \qquad Tr_T\to\infty,
    \qquad \rho(\boldsymbol A_T)<1,\\
    \boldsymbol B_T:=\frac{\boldsymbol I-\boldsymbol A_T}{r_T}
       \to\boldsymbol B,
    \qquad
    T r_T\boldsymbol\mu_T\to\boldsymbol\eta
       \in(0,\infty)^d,
    \qquad
    \sup_T\|\boldsymbol D_T\|<\infty,
\end{gathered}
\end{equation}
where $\boldsymbol B$ is a nonsingular $M$-matrix.  The particular matrix norm is immaterial in fixed dimension.  Define
\begin{equation}\label{eq:multivariate-scaled-processes}
    \boldsymbol X_t^T:=r_T^2\boldsymbol N_{Tt}^T,
    \qquad
    \boldsymbol C_t^T:=r_T^2\boldsymbol\Lambda_{Tt}^T.
\end{equation}

For every $\boldsymbol u\in\R_+^d$, let
$\boldsymbol z(\boldsymbol u)$ denote the unique solution in $\R_+^d$ of
\begin{equation}\label{eq:multitype-IG-root}
    \boldsymbol B^{\top}\boldsymbol z(\boldsymbol u)
       +\frac12\boldsymbol z(\boldsymbol u)^{\circ2}
       =\boldsymbol u,
\end{equation}
where the square is componentwise.  Existence and uniqueness are
proved in Lemma~\ref{lem:multitype-family}.  We call the
$\R_+^d$-valued subordinator $\boldsymbol Z$ characterized by
\begin{equation}\label{eq:multitype-IG-LT}
    \E\left[e^{-\langle\boldsymbol u,\boldsymbol Z_t\rangle}\right]
    =\exp\{-t\langle\boldsymbol\eta,
                   \boldsymbol z(\boldsymbol u)\rangle\},
    \qquad \boldsymbol u\in\R_+^d,
\end{equation}
the \emph{multitype inverse-Gaussian subordinator} with parameters
$(\boldsymbol B,\boldsymbol\eta)$.  This terminology refers to the Brownian first-passage construction below; it also avoids conflating \eqref{eq:multitype-IG-LT} with other, inequivalent distributions called multivariate inverse Gaussian in the literature.

\begin{remark}\label{rem:multitype-passage}
Let $W_1,\ldots,W_d$ be independent standard Brownian motions and define
\[
    \boldsymbol{\mathcal X}_{\boldsymbol s}
      :=(W_1(s_1),\ldots,W_d(s_d))^{\top}
        -\boldsymbol B\boldsymbol s,
    \qquad \boldsymbol s\in\R_+^d.
\]
Because $\boldsymbol B$ is a nonsingular $M$-matrix, the off-diagonal coordinates of each constituent L\'evy process are nondecreasing.  Write
$\boldsymbol T_{\boldsymbol x}$ for the multivariate first-passage vector of
$\boldsymbol{\mathcal X}$ to level $-\boldsymbol x$, in the sense of (3.2)-(3.3) in 
\cite{ChaumontMarolleau2020}.  To verify the claimed identification, let
$\boldsymbol{\mathcal X}^{(j)}$ be the constituent L\'evy process run by
$s_j$.  Its Laplace exponent, under the convention of that reference, is
\[
 \varphi_j(\boldsymbol q)
 :=\log\E\left[e^{-\left\langle\boldsymbol q,
                    \boldsymbol{\mathcal X}^{(j)}_1\right\rangle}\right]
 =\left(\boldsymbol B^{\top}\boldsymbol q\right)_j
    +\frac12q_j^2,
 \qquad \boldsymbol q\in\R_+^d.
\]
 The first-passage exponent is the nonnegative inverse of $\boldsymbol\varphi=(\varphi_1,\ldots,\varphi_d)^{\top}$ by \cite[Theorem~3.3]{ChaumontMarolleau2020}.
Indeed, the hypotheses of that inversion theorem hold here:
if $\boldsymbol p=\boldsymbol B^{-\mathsf T}\boldsymbol1$, then
$\boldsymbol p>\boldsymbol0$ and
$\boldsymbol\varphi(\boldsymbol p)
 =\boldsymbol1+\boldsymbol p^{\circ2}/2>\boldsymbol0$.
Consequently, for any $\boldsymbol u\in(0,\infty)^d$,
\begin{equation}\label{eq:multitype-first-passage-LT}
 \E\left[e^{-\langle\boldsymbol u,
                       \boldsymbol T_{\boldsymbol x}\rangle}
          \1_{\{\boldsymbol T_{\boldsymbol x}\in\R_+^d\}}\right]
 =e^{-\langle\boldsymbol x,\boldsymbol z(\boldsymbol u)\rangle},
\end{equation}
because $\boldsymbol\varphi(\boldsymbol z(\boldsymbol u))
=\boldsymbol u$ is exactly \eqref{eq:multitype-IG-root}.  Continuity of
$\boldsymbol z$ and monotone convergence theorem extend this identity to the
boundary of $\R_+^d$.  In particular, sending
$\boldsymbol u\downarrow\boldsymbol0$ and using
$\boldsymbol z(\boldsymbol0)=\boldsymbol0$ shows that
$\boldsymbol T_{\boldsymbol x}$ is finite almost surely.  The stationary
independent increment property in the ordered level parameter
\cite[Proposition~3.1(2)]{ChaumontMarolleau2020}, together with the
continuity at zero of its Laplace exponent, gives a c\`adl\`ag
modification of the ray restriction
$(\boldsymbol T_{t\boldsymbol\eta})_{t\ge0}$.
We take this modification as the $\R_+^d$-valued subordinator.
Substituting
$\boldsymbol x=t\boldsymbol\eta$ into \eqref{eq:multitype-first-passage-LT} gives
\eqref{eq:multitype-IG-LT}; hence this ray restriction is a version of
$\boldsymbol Z$.
\end{remark}

The first-passage construction identifies the candidate macroscopic clock and makes its dependence structure explicit through the matrix $\boldsymbol B$.  The following theorem shows that this clock arises as the joint scaling limit of the Hawkes counts and their compensators throughout the finite-mean calendar-collapse regime.

\begin{theorem}\label{thm:main}
Under \eqref{eq:multivariate-scaling}, for every $L>0$, as $T\to\infty$,
\begin{equation}\label{eq:main-convergence}
    \left(\boldsymbol X^T,\boldsymbol C^T\right)
       \Longrightarrow(\boldsymbol Z,\boldsymbol Z)
\end{equation}
in $\cD([0,L],\R^{2d})$ equipped with the product $M_1$-topology.  Here $\boldsymbol Z$ is the multitype IG subordinator defined by
\eqref{eq:multitype-IG-LT}.  For every coordinate in which $Z_i$ is discontinuous with positive probability, neither corresponding marginal convergence in
\eqref{eq:main-convergence} can hold in the $J_1$-topology.
\end{theorem}

Theorem~\ref{thm:main} is proved by multitype cluster exploration in
Section~\ref{sec:multiproof}.  Its most transparent special case is obtained when the first-order defect matrix is diagonal.

\begin{corollary}\label{cor:diagonal-IG}
If $\boldsymbol B=\operatorname{diag}(b_1,\ldots,b_d)$ with
$b_i>0$, then the coordinates of $\boldsymbol Z$ are independent classical IG subordinators and
\begin{equation}\label{eq:diagonal-IG-LT}
    \E\left[e^{-uZ_{i,t}}\right]
      =\exp\left\{-t\eta_i\left(\sqrt{b_i^2+2u}-b_i\right)\right\},
      \qquad u\ge0.
\end{equation}
\end{corollary}

\begin{proof}
Equation \eqref{eq:multitype-IG-root} separates coordinatewise and gives
$z_i(\boldsymbol u)=\sqrt{b_i^2+2u_i}-b_i$.  The exponent in
\eqref{eq:multitype-IG-LT} is therefore a sum of one-dimensional exponents, which proves independence and \eqref{eq:diagonal-IG-LT}.
\end{proof}

We now recover the scalar model used in the microstructure interpretation and in the explicit Borel calculations below.  Let
$\phi:\R_+\to\R_+$ be measurable and satisfy
\begin{equation}\label{eq:kernel-assumption}
    \int_0^\infty \phi(s)\dd s=1,
    \qquad
    m_\phi:=\int_0^\infty s\phi(s)\dd s<\infty.
\end{equation}
For $a_T\in(0,1)$ take $\phi_T=a_T\phi$, set $r_T:=1-a_T$, and let
\begin{equation}\label{eq:hawkes}
    \lambda_t^T
    =\mu_T+a_T\int_{(0,t)}\phi(t-s)\dd N_s^T,
    \qquad
    \Lambda_t^T:=\int_0^t\lambda_s^T\dd s.
\end{equation}
The following result is the one-dimensional special case of Corollary~\ref{cor:diagonal-IG} and estalishes the joint convergence of rescaled processes 
\begin{equation}\label{eq:scaled-processes}
    X_t^T:=r_T^2N_{Tt}^T
    \quad \mbox{and}\quad 
    A_t^T:=r_T^2\Lambda_{Tt}^T,\qquad t\geq 0.
\end{equation}

\begin{corollary} \label{cor:scalar-IG}
Under \eqref{eq:kernel-assumption}, if $r_T\to0$, $Tr_T\to\infty$ and $\vartheta_T:=T\mu_Tr_T \to\vartheta\in(0,\infty)$, we have
 \begin{equation}\label{eq:201}
 (X^T,A^T) \Longrightarrow (Y,Y)
 \end{equation}
 in $\cD([0,L],\R^{2})$ equipped with the product $M_1$-topology as $T\to\infty$, where $Y$ is a IG subordinator with $Y_0=0$ and Laplace tranform
\begin{equation}\label{eq:IG-LT}
    \E\left[e^{-uY_t}\right]
      =\exp\left\{-t\vartheta\left(\sqrt{1+2u}-1\right)\right\},
      \qquad u\ge0. 
\end{equation}
 In particular, $Y$ has L\'evy measure and one-time law  
\begin{align}
    \nu_{\mathrm{IG}}(\dd x)
      =\frac{\vartheta}{\sqrt{2\pi}}x^{-3/2}
        e^{-x/2}\1_{(0,\infty)}(x)\dd x
        \quad \mbox{and}\quad 
         Y_t\sim\IG\left(\vartheta t,\vartheta^2t^2\right),
    \qquad t>0. \label{eq:IG-Levy}
\end{align} 
Neither marginal convergence of \eqref{eq:201} can hold in the $J_1$-topology.
\end{corollary}

\begin{proof}
Here $\boldsymbol A_T=a_T$, $\boldsymbol B_T=1$,
$\boldsymbol D_T=a_Tm_\phi$, and $\boldsymbol\eta=\vartheta$.
Theorem~\ref{thm:main} and Corollary~\ref{cor:diagonal-IG} apply.  The
L\'evy density follows by integrating $1-e^{-ux}$, and the stated IG
parameters follow from the convention in which $\IG(m,\lambda)$ has mean
$m$ and variance $m^3/\lambda$.

The measure in \eqref{eq:IG-Levy} has infinite mass, so $Y$ has
discontinuous paths almost surely.  The $J_1$ assertion therefore
follows from Theorem~\ref{thm:main}.
\end{proof}

The square-root regime is the scalar specialization in which the natural family-mass scale $r_T^{-2}$ is asymptotic to a constant multiple of $T$.  Corollary~\ref{cor:scalar-IG} then takes the familiar $T^{-1}$ normalization and yields the following parameterized IG limit.

\begin{corollary} \label{cor:square-root}
Assume \eqref{eq:square-root-scaling} and \eqref{eq:kernel-assumption}.  For every $L>0$, as $T\to\infty$,
\[
    \left(\frac{N^T_{T\cdot}}{T},
    \frac{\Lambda^T_{T\cdot}}{T}\right)
    \Longrightarrow (Y^{\gamma,\eta},Y^{\gamma,\eta}),
\]
in $\cD([0,L],\R^{2})$ equipped with the product $M_1$-topology, where the limit process $Y^{\gamma,\eta}$ is uniquely determined by
\begin{equation}\label{eq:square-root-IG-LT}
    \E\left[e^{-uY_t^{\gamma,\eta}}\right]
    =\exp\left\{-t\eta
      \left(\sqrt{\gamma^2+2u}-\gamma\right)\right\}.
\end{equation}
Neither marginal convergence can hold in $J_1$.
\end{corollary}

\begin{proof}
Here $Tr_T\to\infty$, $\vartheta_T=T\mu_Tr_T\to\gamma\eta$, and $Tr_T^2\to\gamma^2$.  Corollary~\ref{cor:scalar-IG} and
\[
    \frac{N^T_{Tt}}{T}
    =\frac{r_T^2N^T_{Tt}}{Tr_T^2}
\]
give the result.  Scaling the IG clock in \eqref{eq:IG-LT} by
$\gamma^{-2}$ gives \eqref{eq:square-root-IG-LT}.
Since $Tr_T^2\to\gamma^2>0$, the $J_1$ assertion is also preserved
by this deterministic rescaling and follows from
Corollary~\ref{cor:scalar-IG}.
\end{proof}

\begin{remark}\label{rem:matching}
For $g_0>0$ and $\nu\ne0$, the IG process in
\cite[Theorem~3.1 and Definition~2.2]{AbiJaberAttalRosenbaum2025}
has process parameters
$\left(\frac{g_0}{1+\lambda},\frac{g_0^2}{\nu^2}\right)$.
In particular, its marginal at time $t>0$ has law
\[
    \IG\left(\frac{g_0t}{1+\lambda},
                   \frac{g_0^2t^2}{\nu^2}\right).
\]
Taking \eqref{eq:parameter-map} into Corollary~\ref{cor:square-root}
gives exactly these process and marginal parameters.  Thus the Volterra singular limit and the direct branching limit meet at the same IG law and the same Brownian first-passage structure.
\end{remark}

Subordinating an independent Brownian motion by the scalar IG clock
gives the following normal inverse Gaussian process.

\begin{corollary}\label{cor:NIG}
Let $W$ be a standard Brownian motion independent of $Y^{\gamma,\eta}$ and let $\varsigma>0$.  Then
\[
    P_t:=\varsigma W_{Y_t^{\gamma,\eta}},\qquad t\ge0,
\]
is a symmetric normal inverse Gaussian L\'evy process with characteristic function
\begin{equation}\label{eq:NIG-characteristic}
    \E\left[e^{\mathrm{i}uP_t}\right]
    =\exp\left\{-t\eta
       \left(\sqrt{\gamma^2+\varsigma^2u^2}-\gamma\right)\right\}.
\end{equation}
\end{corollary}

\begin{proof}
Conditioning on $Y^{\gamma,\eta}$ and using \eqref{eq:square-root-IG-LT} gives
\[
    \E\left[e^{\mathrm{i}uP_t}\right]
    =\E\left[e^{-\varsigma^2u^2Y_t^{\gamma,\eta}/2}\right],
\]
which is \eqref{eq:NIG-characteristic}.  
\end{proof}

\section{Beyond the high-intensity regime and a general \texorpdfstring{$M_1$}{M1} upgrade}\label{sec:general}

Theorem~\ref{thm:main} is genuinely multivariate, whereas the hyper-rough limit in \cite[Theorem~3.1]{AbiJaberAttalRosenbaum2025} is one-dimensional.  In this section we compare our result with the more general $d$-variate, possibly time-inhomogeneous framework of \cite[Theorems~2.9, 2.12, and~2.13]{Xu2024Scaling}.  The formulation below is inherited from that work.  Our contributions here are to weaken its high-intensity assumption under common tilted stability, provide the Riccati existence and identification argument, strengthen convergence of the monotone count and compensator to product $M_1$, and locate the multitype IG specialization precisely.  In particular, the diagonal-atom condition of \cite{Xu2024Scaling} applies when $\boldsymbol B$ is diagonal; the coupled non-diagonal case of Theorem~\ref{thm:main} is covered by the cluster proof in Section~\ref{sec:multiproof}.

At the scalar level, \cite[Proposition~2.4]{Zitkovic2026Feller} establishes uniqueness for bounded nonnegative solutions of the Riccati--Volterra equation with a positive geometrically infinitely divisible (PGID) memory measure, including an atom at zero.  Theorem~\ref{thm:relaxed-Xu} below treats multivariate complex Fourier--Laplace test functions, identifies the joint limit of the Hawkes counts, compensators, and compensated martingales, and proves local $L^p$ convergence of the prelimit Riccati solutions for $1\le p<\infty$.  The scalar parameter limits and their $M_1$ interpretation in \cite[Proposition~3.6 and Remark~3.8]{Zitkovic2026Feller} concern cumulative Feller random measures; Corollary~\ref{cor:general-M1} gives joint product-$M_1$ convergence of the original Hawkes counts and compensators.

For the specialization relevant to this paper, $\theta_T=(Tr_T^2)^{-1}$.  Consequently, the original assumption $\theta_T\to\infty$ covers only the subregion $Tr_T^2\to0$; it excludes both the square-root boundary $Tr_T^2\to\gamma^2$ and the region $Tr_T^2\to\infty$.  Replacing it by $T\theta_T=r_T^{-2}\to\infty$ allows the general theory to cover the entire calendar-collapse region.  The formulation below also fixes the distinction between microscopic intensity and the number of events accumulated over a macroscopic interval.

\subsection{The intrinsic activity scale}

For a fixed dimension $d$, let $\boldsymbol N_n$ be a $d$-variate linear Hawkes process \cite[equations~(2.1)--(2.2)]{Xu2024Scaling}.  For every $n$, let
\[
    \boldsymbol\mu_n\in L^1_{\mathrm{loc}}\left(\R_+;\R_+^d\right),
    \qquad
    \boldsymbol\phi_n\in L^1\left(\R_+;\R_+^{d\times d}\right)
\]
be deterministic and nonnegative.  Here $\mu_{n,i}(t)$ is the possibly time-dependent exogenous intensity of component $i$, while $\phi_{n,ij}(u)$ is the excitation exerted on component $i$ by an event of component $j$ after a lag $u$.  We set $\boldsymbol\phi_n(0)=\boldsymbol0$.  The predictable intensity is
\begin{equation}\label{eq:general-Hawkes}
    \boldsymbol\lambda_n(t)
    =\boldsymbol\mu_n(t)
     +\int_{[0,t)}\boldsymbol\phi_n(t-s)\dd\boldsymbol N_n(s),
    \qquad t\ge0,
\end{equation}
Define the compensator and compensated martingale by
\[
    \boldsymbol C_n(t):=\int_0^t\boldsymbol\lambda_n(s)\dd s,
    \qquad
    \widetilde{\boldsymbol N}_n(t):=\boldsymbol N_n(t)-\boldsymbol C_n(t).
\]
Set
\[
    \boldsymbol A_n:=\int_0^\infty\boldsymbol\phi_n(s)\dd s
\]
(the integral is entrywise), and let
$\boldsymbol R_n:=\sum_{k=1}^{\infty}\boldsymbol\phi_n^{*k}$ be the resolvent.  Instead of starting from a microscopic intensity scale $\theta_n$, introduce directly an activity scale $b_n>0$ with $b_n\to\infty$ and set
\begin{equation}\label{eq:general-rescaling}
\begin{aligned}
    \boldsymbol{\mathcal A}^{(n)}(t)
       &:=\frac{\boldsymbol C_n(nt)}{b_n^2},
    &\qquad
    \boldsymbol{\mathcal N}^{(n)}(t)
       &:=\frac{\boldsymbol N_n(nt)}{b_n^2},\\
    \boldsymbol{\mathcal M}^{(n)}(t)
       &:=\frac{\widetilde{\boldsymbol N}_n(nt)}{b_n},
    &
    \boldsymbol R^{(n)}(t)
       &:=\frac{n}{b_n}\boldsymbol R_n(nt),\\
    \boldsymbol\mu^{(n)}(t)
       &:=\frac{n}{b_n}\boldsymbol\mu_n(nt),
    &
    \boldsymbol H^{(n)}(t)
       &:=\frac{\boldsymbol\mu^{(n)}(t)}{b_n}
          +\boldsymbol R^{(n)}*\boldsymbol\mu^{(n)}(t).
\end{aligned}
\end{equation}
The parameterization in \cite[equations~(2.8)--(2.9)]{Xu2024Scaling} is recovered through
\begin{equation}\label{eq:activity-reparameterization}
    \theta_n=\frac{b_n^2}{n},
    \qquad \sqrt{n\theta_n}=b_n,
    \qquad \sqrt{n/\theta_n}=\frac{n}{b_n},
    \qquad n\theta_n=b_n^2.
\end{equation}
Thus $\theta_n$ measures the microscopic intensity scale, whereas $b_n^2$ measures the total activity accumulated over a time interval of length $n$.

Let $\boldsymbol K$ be a nonnegative $d\times d$ matrix with spectral radius one, and let $\ell$ be the algebraic multiplicity of its eigenvalue one.  Set $\mathtt I:=\{1,\ldots,\ell\}$ and $\mathtt J:=\{\ell+1,\ldots,d\}$.  We impose explicitly the admissibility hypotheses of \cite[Definition~2.5 and Assumption~2.8]{Xu2024Scaling}.  Let $\boldsymbol\Phi\in C(\R_+;\R^{d\times d})$.  This function is called $\boldsymbol K$-admissible if there exist an orthogonal matrix $\boldsymbol Q$, a real matrix $\boldsymbol U$, and a constant $\lambda^+\ge0$ such that
\begin{equation}\label{eq:general-Schur-decomposition}
    \boldsymbol K=\boldsymbol Q\boldsymbol U\boldsymbol Q^{\top},
    \qquad
    \boldsymbol U_{\mathtt I\mathtt I}=\mathrm{Id}_{\ell},
    \qquad
    \boldsymbol U_{\mathtt J\mathtt I}=\boldsymbol0,
\end{equation}
and the following invertibility condition holds for an $\ell\times\ell$ matrix $\boldsymbol G(\lambda)$:
\begin{equation}\label{eq:general-admissibility}
\begin{gathered}
    \boldsymbol G(\lambda)
    :=
    \left(\boldsymbol Q^{\top}
          \boldsymbol\Phi(\lambda)\boldsymbol Q
    \right)_{\mathtt I\mathtt I}
    +\boldsymbol U_{\mathtt I\mathtt J}
    \left(\mathrm{Id}_{d-\ell}
          -\boldsymbol U_{\mathtt J\mathtt J}\right)^{-1}
    \left(\boldsymbol Q^{\top}
          \boldsymbol\Phi(\lambda)\boldsymbol Q
    \right)_{\mathtt J\mathtt I},\\
    \det\boldsymbol G(\lambda)\ne0,
    \qquad \lambda\ge\lambda^+.
\end{gathered}
\end{equation}
Empty blocks are omitted when $\ell=d$.  The stable block $\boldsymbol U_{\mathtt J\mathtt J}$ has spectral radius at most one and no eigenvalue equal to one, so $\mathrm{Id}_{d-\ell}-\boldsymbol U_{\mathtt J\mathtt J}$ is invertible.  This block formulation permits real $2\times2$ Schur blocks for complex eigenvalues.
In addition to algebraic admissibility, we require a common exponential
tilt at which the kernels are subcritical: for some $\lambda_0\ge\lambda^+$,
\begin{equation}\label{eq:general-tilted-stability}
    \rho\!\left(\mathcal L_{\boldsymbol\phi_n}(\lambda_0/n)\right)<1
    \qquad\text{for all sufficiently large }n.
\end{equation}
This condition ensures positivity of the resolvent transforms.
In the scalar case it excludes $\Phi\equiv-1$, which is algebraically
admissible but would require a negative potential measure.
Under the convergence assumptions below,
Lemma~\ref{lem:positive-resolvent-limit} proves that the associated
$(\boldsymbol K,\boldsymbol\Phi)$-potential measure $\boldsymbol\Pi$
is the unique nonnegative, locally finite matrix-valued measure with
entrywise Laplace transform
\begin{equation}\label{eq:general-potential}
    \mathcal L_{\boldsymbol\Pi}(\lambda)
    =\boldsymbol Q
      \begin{pmatrix}
        \left(\boldsymbol G(\lambda)\right)^{-1}
        &\left(\boldsymbol G(\lambda)\right)^{-1}
         \boldsymbol U_{\mathtt I\mathtt J}
         \left(\mathrm{Id}_{d-\ell}
               -\boldsymbol U_{\mathtt J\mathtt J}\right)^{-1}\\
        \boldsymbol0&\boldsymbol0
      \end{pmatrix}
      \boldsymbol Q^{\top},
    \qquad \lambda\ge\lambda_0.
\end{equation}
We further assume that the atom $\boldsymbol\Pi(\{0\})$ is a diagonal matrix, namely,
\begin{equation}\label{eq:general-diagonal-atom}
    \Pi_{ij}(\{0\})=0,\qquad i\ne j.
\end{equation}
The admissibility, tilted-stability, and diagonal-atom requirements are
distinct from the three convergence assumptions below.
The assumptions on kernel criticality and cumulative exogenous activity, written in the intrinsic scale, are
\begin{align} \label{eq:general-mass}
    &\boldsymbol A_n\to\boldsymbol K, \quad 
        b_n\left[\boldsymbol K-
       \mathcal L_{\boldsymbol\phi_n}(\lambda/n)\right]
       \to\boldsymbol\Phi(\lambda),
       \qquad \lambda\ge0
\end{align}
and 
\begin{align} \label{eq:general-exogenous}
    &\mathcal I_{\boldsymbol H^{(n)}}(t)
       :=\int_0^t\boldsymbol H^{(n)}(s)\dd s
       \to\boldsymbol\Upsilon(t),
\end{align}
at every continuity point of a nondecreasing function $\boldsymbol\Upsilon\in\cD(\R;\R_+^d)$ that is zero on $(-\infty,0)$.  We extend the cumulative functions on the left-hand side by zero to negative times and use the Stieltjes convention $\boldsymbol\Upsilon(0-)=\boldsymbol0$.  Here $\boldsymbol H^{(n)}$ is the resolvent-adjusted exogenous intensity, $\mathcal I_{\boldsymbol H^{(n)}}$ is its cumulative exogenous activity, and $\boldsymbol\Upsilon$ is the limiting cumulative exogenous activity.  The symbol $\mathcal L$ denotes the entrywise Laplace transform, and all vector squares below are componentwise.

The common tilt makes the Schur-complement limit a transform of a
positive measure.  The following lemma establishes this fact directly
and provides the resolvent bounds used in the Riccati argument.

\begin{lemma} \label{lem:positive-resolvent-limit}
Assume $b_n\to\infty$, 
\eqref{eq:general-Schur-decomposition}-\eqref{eq:general-tilted-stability} and \eqref{eq:general-mass}.  Then the nonnegative measure
$\boldsymbol\Pi$ in \eqref{eq:general-potential} exists uniquely and
\begin{equation}\label{eq:relaxed-resolvent-convergence}
    \boldsymbol R^{(n)}(t)\dd t \to\boldsymbol\Pi(\dd t),
\end{equation}
entrywise vaguely on $\mathbb{R}_+$ as $n\to\infty$.  Moreover,  its exponentially tilted transforms satisfy that as $n\to\infty$, 
\begin{equation}\label{eq:positive-resolvent-transform}
 \mathcal L_{\boldsymbol R^{(n)}}(\lambda)
 =\frac1{b_n}
       \left(\boldsymbol I-
          \mathcal L_{\boldsymbol\phi_n}(\lambda/n)\right)^{-1}
       \mathcal L_{\boldsymbol\phi_n}(\lambda/n)
 \to \mathcal L_{\boldsymbol\Pi}(\lambda),
 \qquad \lambda\ge\lambda_0.
\end{equation}
\end{lemma}

\begin{proof}
For $\lambda\ge\lambda_0$, let
$A_n(\lambda):=\mathcal L_{\boldsymbol\phi_n}(\lambda/n)$ and
$F_n(\lambda):=b_n(\boldsymbol K-A_n(\lambda))$.
Positivity gives $0\le A_n(\lambda)\le A_n(\lambda_0)$.
The spectral radius is monotone on nonnegative matrices, so
\eqref{eq:general-tilted-stability} and the Neumann series give
$\left(\boldsymbol I-A_n(\lambda)\right)^{-1}\ge0$ for large $n$.
Tonelli's theorem applied to the resolvent series then gives the
identity in \eqref{eq:positive-resolvent-transform}.

Let $\widetilde F_n:=\boldsymbol Q^{\top}F_n(\lambda)\boldsymbol Q$ and
$E_n:=\mathrm{Id}_{d-\ell}-\boldsymbol U_{\mathtt J\mathtt J}
       +b_n^{-1}\widetilde F_{n,\mathtt J\mathtt J}$.
In the $\boldsymbol Q$ coordinates,
\[
 \boldsymbol I-A_n(\lambda)
 \quad\text{has blocks}\quad
 \begin{pmatrix}
 b_n^{-1}\widetilde F_{n,\mathtt I\mathtt I}
   &-\boldsymbol U_{\mathtt I\mathtt J}
          +b_n^{-1}\widetilde F_{n,\mathtt I\mathtt J}\\
 b_n^{-1}\widetilde F_{n,\mathtt J\mathtt I}&E_n
 \end{pmatrix}.
\]
The block $E_n$ is invertible for large $n$.  Its Schur complement
$S_n$ satisfies
\[
 b_n S_n
 =\widetilde F_{n,\mathtt I\mathtt I}
  +\left(\boldsymbol U_{\mathtt I\mathtt J}
             -b_n^{-1}\widetilde F_{n,\mathtt I\mathtt J}\right)
      E_n^{-1}\widetilde F_{n,\mathtt J\mathtt I}
 \longrightarrow\boldsymbol G(\lambda).
\]
The block inverse formula therefore gives
\[
 \frac1{b_n}\left(\boldsymbol I-A_n(\lambda)\right)^{-1}
 \longrightarrow P(\lambda),
\]
where $P(\lambda)$ is the right-hand side of
\eqref{eq:general-potential}.  The case $\ell=d$ follows by
inverting $F_n(\lambda)$ directly.  Since
\[
 \frac1{b_n}\left(\boldsymbol I-A_n(\lambda)\right)^{-1}A_n(\lambda)
 =\frac1{b_n}\left(\boldsymbol I-A_n(\lambda)\right)^{-1}
     -\frac{\boldsymbol I}{b_n},
\]
the transforms in \eqref{eq:positive-resolvent-transform} converge to
$P(\lambda)$.  This function is continuous for $\lambda\ge\lambda_0$
by \eqref{eq:general-admissibility}.

For each entry, consider the finite positive measure
$e^{-\lambda_0 t}R_{ij}^{(n)}(t)\dd t$.
Its mass converges to $P_{ij}(\lambda_0)$, and its Laplace transform
at $a\ge0$ converges to $P_{ij}(\lambda_0+a)$.
If the limiting mass is positive, normalize these measures by their
masses and apply the extended continuity theorem for Laplace transforms
\cite[Theorem~6.23]{Kallenberg2021}; if it is zero, the measures converge
to zero in total variation.  Thus they converge weakly to a finite
positive measure $\nu_{ij}$.  Let
$\Pi_{ij}(\dd t):=e^{\lambda_0 t}\nu_{ij}(\dd t)$.
These measures are locally finite, have transform $P$, and satisfy
\eqref{eq:relaxed-resolvent-convergence}.  Uniqueness follows from
uniqueness of the Laplace transform of the finite tilted measures
\cite[Theorem~6.3]{Kallenberg2021}.
\end{proof}

The next theorem establishes the scaling limit in the activity scale
$b_n$, under common tilted stability.  Its direct proof follows the
local Riccati estimates below.

\begin{theorem} \label{thm:relaxed-Xu}
Assume $b_n\to\infty$ and \eqref{eq:general-mass}--\eqref{eq:general-exogenous},
together with \eqref{eq:general-Schur-decomposition}--\eqref{eq:general-admissibility},
the common tilted-stability condition \eqref{eq:general-tilted-stability},
and the diagonal-atom condition \eqref{eq:general-diagonal-atom}.
Then the scaling conclusions of
\cite[Theorems~2.9, 2.12, and~2.13]{Xu2024Scaling} hold in this setting.
More precisely, there exist a nondecreasing process $\boldsymbol\Xi$ and a
square-integrable martingale $\boldsymbol M$ with the following properties.
\begin{enumerate}[label=(\roman*),leftmargin=*]
\item \emph{Process convergence.}
The rescaled processes satisfy
\begin{equation}\label{eq:general-S-limit}
   \left(\boldsymbol{\mathcal A}^{(n)},
         \boldsymbol{\mathcal N}^{(n)},
         \boldsymbol{\mathcal M}^{(n)}\right)
   \Longrightarrow
   (\boldsymbol\Xi,\boldsymbol\Xi,\boldsymbol M),
\end{equation}
in Jakubowski's $S$-topology.  The finite-dimensional distributions
also converge at every positive continuity point of
$\boldsymbol\Upsilon$.

\item \emph{Riccati equation and regularity.}
For row-vector test functions
\[
 \boldsymbol f\in C^1\left(\R_+;\mathbb C_-^d\right),
 \qquad
 \boldsymbol h\in C^1\left(\R_+;\mathrm i\R^d\right),
 \qquad
 \mathbb C_-:=\left\{z\in\mathbb C:\operatorname{Re}z\le0\right\},
\]
the prelimit solutions in \eqref{eq:local-prelimit-Riccati} converge
in $L^p(0,L)$, for every $L>0$ and $1\le p<\infty$, to the unique global
solution $\boldsymbol V\in\cD(\R_+;\mathbb C_-^d)$ of
\begin{equation}\label{eq:general-Riccati}
    \boldsymbol V=\boldsymbol W*\boldsymbol\Pi,
    \qquad
    \boldsymbol W
       =\boldsymbol f+\frac12(\boldsymbol V+\boldsymbol h)^2.
\end{equation}
The solution $\boldsymbol V$ is continuous if
$\boldsymbol f(0)=\boldsymbol h(0)=\boldsymbol0$, or if
$t\mapsto\boldsymbol\Pi([0,t])$ is continuous on $\R_+$.

\item \emph{Characterization of the limiting processes.}
For the test functions and solution in \textup{(ii)}, the Fourier--Laplace
functional satisfies
\begin{equation}\label{eq:general-FLF}
    \E\!\left[
       \exp\left\{\boldsymbol f*\dd\boldsymbol\Xi(t)
                 +\boldsymbol h*\dd\boldsymbol M(t)\right\}
       \right]
    =\exp\left\{\boldsymbol W*\dd\boldsymbol\Upsilon(t)\right\}.
\end{equation}
Here row vectors multiply matrix-valued kernels from the left, and, for
example,
\[
 \boldsymbol f*\dd\boldsymbol\Xi(t)
 :=\int_{[0,t]}\boldsymbol f(t-s)\dd\boldsymbol\Xi(s).
\]
The limiting processes are extended by zero to negative times, and
these integrals include any initial jump.  For smooth $\boldsymbol h$,
the martingale integral uses the same convention, as expressed by
the integration-by-parts formula in the proof below.
Moreover,
\begin{equation}\label{eq:general-Volterra}
    \boldsymbol\Xi
       =\boldsymbol\Upsilon+\boldsymbol\Pi*\boldsymbol M.
\end{equation}

\item \emph{Moment bounds and stochastic continuity.}
For every $t\ge0$,
\begin{equation}\label{eq:general-limit-moments}
 \E\left[|\boldsymbol\Xi(t)|_1\right]\le|\boldsymbol\Upsilon(t)|_1,
 \qquad
 \E\left[|\boldsymbol M(t)|_1^2\right]\le d|\boldsymbol\Upsilon(t)|_1.
\end{equation}
If $\boldsymbol\Upsilon$ is continuous on $\R_+$, then
$(\boldsymbol\Xi,\boldsymbol M)$ is stochastically continuous.
If it is continuous on $\R$, then also
$\boldsymbol\Xi(0)=\boldsymbol M(0)=\boldsymbol0$.
\end{enumerate}
\end{theorem}

Equivalently, in the original parameterization the condition in Theorem~\ref{thm:relaxed-Xu} is
\begin{equation}\label{eq:relaxed-theta}
    \theta_n>0
    \quad \mbox{and}\quad  
    n\theta_n\to\infty,
\end{equation}
in place of $\theta_n\to\infty$.  The normalizations and simplicity
of the coordinate counts give
\begin{equation}\label{eq:count-compensator-difference}
 \boldsymbol{\mathcal N}^{(n)}-\boldsymbol{\mathcal A}^{(n)}
       =\frac{\boldsymbol{\mathcal M}^{(n)}}{b_n}
       \quad \mbox{and}\quad 
 \sup_{0<t\le L}\left|\Delta\mathcal M_i^{(n)}(t)\right|
       \le b_n^{-1}\to 0,\quad i=1,\ldots,d.
\end{equation}
The remainder in \eqref{eq:Riccati-Taylor} vanishes under the same
condition.  The proof uses the observation scale only through
\eqref{eq:general-rescaling} and
\eqref{eq:general-mass}--\eqref{eq:general-exogenous}; it applies to
any sequence of real observation scales tending to infinity.
We therefore use $n=T$ in the IG specialization below.

\subsection{Local convergence in the Riccati identification}

We combine the stochastic compactness argument of
\cite[Lemma~5.4]{Xu2024Scaling} with positive-resolvent bounds,
local convolution limits, and a fixed logarithm branch.  These steps
establish the Riccati existence and transform identification in
Theorem~\ref{thm:relaxed-Xu}, including the case of an atom at zero.

For a vector $x\in\mathbb C^d$, let
$|x|_1:=\sum_{i=1}^d|x_i|$, and use the associated integral norms.
For a nonnegative, locally finite matrix-valued measure $K$ on
$[0,\infty)$, let
\[
 m_K(L):=\sum_{i=1}^d\sum_{j=1}^dK_{ij}([0,L]),
 \qquad
 (g*K)(t):=\int_{[0,t]}g(t-s)K(\dd s).
\]
Functions are extended by zero to negative times.  Tonelli's theorem
gives the local convolution inequality
\begin{align}\label{eq:local-measure-Young}
 \|g*K\|_{L^1(0,L)}
 &\le\sum_{i=1}^d\sum_{j=1}^d
       \int_{[0,L]}\left(\int_s^L|g_i(t-s)|\dd t\right)K_{ij}(\dd s)
       \notag\\
 &\le m_K(L)\|g\|_{L^1(0,L)}.
\end{align}
The corresponding inequality for a nonnegative column-vector measure
uses the sum of its coordinate masses.

The next lemma permits a limiting atom at zero and controls convolution
with a fixed, locally integrable function.

\begin{lemma} 
\label{lem:local-convolution-vague}
Let $K_n$ and $K$ be nonnegative, locally finite matrix-valued measures on
$[0,\infty)$ such that $K_n\to K$ entrywise vaguely.  For every
row-vector function $g\in L^1_{\mathrm{loc}}(\R_+;\mathbb C^d)$ and
every $L>0$,
\begin{equation}\label{eq:local-fixed-convolution-limit}
    \|g*K_n-g*K\|_{L^1(0,L)}\longrightarrow0.
\end{equation}
The same assertion holds for column-vector measures.
\end{lemma}

\begin{proof}
Fix $L>0$.  A continuous compactly supported majorant of
$\1_{[0,L]}$, positivity, and vague convergence give
\[
    C_L:=\sup_n m_{K_n}(L)+m_K(L)<\infty.
\]
Choose $q\in C_c((0,L);\mathbb C^d)$ with
$\|g-q\|_{L^1(0,L)}<\varepsilon$, and extend $q$ by zero to $\R$.
For every $t\in[0,L]$, the function $s\mapsto q(t-s)$ is continuous
and compactly supported on $[0,\infty)$.  Thus
$(q*K_n)(t)\to(q*K)(t)$, while
\[
    \left|(q*K_n)(t)-(q*K)(t)\right|_1
       \le C_L\|q\|_{L^\infty(\R)}.
\]
The dominated convergence theorem gives convergence in $L^1(0,L)$
for this difference.  Applying \eqref{eq:local-measure-Young} to $g-q$
then yields
\[
 \limsup_{n\to\infty}\left\|g*K_n-g*K\right\|_{L^1(0,L)}
       \le C_L\|g-q\|_{L^1(0,L)}
       \le C_L\varepsilon.
\]
Letting $\varepsilon\downarrow0$ proves
\eqref{eq:local-fixed-convolution-limit}.  For column-vector measures,
the proof uses the same scalar convolution inequality.
\end{proof}

The positive tilted resolvent gives a uniform bound for the prelimit
Riccati solutions.  This bound will both control the Taylor remainder
and keep the auxiliary transforms in a fixed logarithm branch.

\begin{lemma} \label{lem:uniform-Riccati-bound}
Under the hypotheses of Lemma~\ref{lem:positive-resolvent-limit}, let
$\boldsymbol V^{(n)}$ be the prelimit $\mathbb C_-^d$-valued solution
in \eqref{eq:local-prelimit-Riccati}, for locally bounded test functions
$\boldsymbol f$ and $\boldsymbol h$ with values in $\mathbb C_-^d$ and
$\mathrm i\R^d$, respectively.  For every $L>0$,
\begin{equation}\label{eq:uniform-Riccati-bound}
 \sup_{n\ge n_0}
   \left\|\boldsymbol V^{(n)}\right\|_{L^\infty(0,L)}
 \le C_L\left(\|\boldsymbol f\|_{L^\infty(0,L)}
                    +\|\boldsymbol h\|_{L^\infty(0,L)}^2\right)
\end{equation}
for some $n_0$ and $C_L$ independent of the test functions.
\end{lemma}

\begin{proof}
Let $\kappa_n(t):=n\boldsymbol\phi_n(nt)/b_n$.
The prelimit equation, equivalently
\cite[equations~(4.6)--(4.7)]{Xu2024Scaling}, is
\[
 \boldsymbol V^{(n)}
 =\left[\boldsymbol f
        +b_n^2\left(e^{\boldsymbol h/b_n}-\boldsymbol1
                                  -\boldsymbol h/b_n\right)
        +b_n^2 e^{\boldsymbol h/b_n}
                    \left(e^{\boldsymbol V^{(n)}/b_n}-\boldsymbol1\right)
   \right]*\kappa_n.
\]
For $\operatorname{Re}z\le0$, the integral formulas for the exponential
give $|e^z-1|\le|z|$ and $|e^z-1-z|\le|z|^2/2$.
Let
\[
 a_i:=\|f_i\|_{L^\infty(0,L)}
              +\frac12\|h_i\|_{L^\infty(0,L)}^2,
 \qquad
 v_{n,i}:=\sup_{0\le t\le L}e^{-\lambda_0t}\left|V_i^{(n)}(t)\right|.
\]
Since $|e^{h_i/b_n}|=1$, positivity of the kernel yields the
row-vector inequality
\[
 v_n\le\frac{a}{b_n}A_n(\lambda_0)+v_n A_n(\lambda_0),
 \qquad A_n(\lambda_0)=\mathcal L_{\boldsymbol\phi_n}(\lambda_0/n).
\]
Multiplying by the nonnegative inverse furnished by
\eqref{eq:general-tilted-stability} gives
\[
 v_n\le \frac{a}{b_n}A_n(\lambda_0)
                       (\boldsymbol I-A_n(\lambda_0))^{-1}
      =a\mathcal L_{\boldsymbol R^{(n)}}(\lambda_0).
\]
The transforms on the right are bounded by
\eqref{eq:positive-resolvent-transform}.  Multiplication by
$e^{\lambda_0L}$ proves \eqref{eq:uniform-Riccati-bound}.
\end{proof}

Almost-everywhere convergence together with
Lemma~\ref{lem:uniform-Riccati-bound} gives convergence in every finite
local $L^p$ norm.  The next proposition provides the two convolution
passages used to identify the limiting equation and transform.

\begin{proposition} 
\label{prop:local-Riccati-repair}
Let $b_n>0$ tend to infinity, and assume
\eqref{eq:relaxed-resolvent-convergence}.  For row-vector functions
$\boldsymbol f\in L^\infty_{\mathrm{loc}}(\R_+;\mathbb C_-^d)$ and
$\boldsymbol h\in L^\infty_{\mathrm{loc}}(\R_+;\mathrm i\R^d)$, suppose
\begin{equation}\label{eq:local-prelimit-Riccati}
 \boldsymbol V^{(n)}
       =\boldsymbol W^{(n)}*\boldsymbol R^{(n)},\qquad
 \boldsymbol W^{(n)}
       =\boldsymbol f+b_n^2\mathcal W
           \left(\frac{\boldsymbol V^{(n)}+\boldsymbol h}{b_n}\right),
 \qquad
 \mathcal W(z):=(e^{z_i}-1-z_i)_{i=1}^d,
\end{equation}
where $\boldsymbol V^{(n)}$ takes values in $\mathbb C_-^d$,
$\sup_n\left\|\boldsymbol V^{(n)}\right\|_{L^\infty(0,L)}<\infty$ for every $L>0$,
and $\boldsymbol V^{(n)}\to\boldsymbol V$ almost everywhere.
Let $\boldsymbol W:=\boldsymbol f+(\boldsymbol V+\boldsymbol h)^2/2$.
Then, for every $L>0$ and $1\le p<\infty$,
\begin{equation}\label{eq:local-Riccati-Lp}
 \left\|\boldsymbol V^{(n)}-\boldsymbol V\right\|_{L^p(0,L)}
 +\left\|\boldsymbol W^{(n)}-\boldsymbol W\right\|_{L^p(0,L)}
 \to 0,
\end{equation}
and $\boldsymbol V=\boldsymbol W*\boldsymbol\Pi$ almost everywhere.
Moreover, if the nonnegative column-vector measures
$\nu_n(\dd s):=\boldsymbol H^{(n)}(s)\dd s$ converge vaguely to
$\nu(\dd s):=\dd\boldsymbol\Upsilon(s)$, then
\begin{equation}\label{eq:local-FLF-L1}
 \left\|\boldsymbol W^{(n)}*\nu_n-\boldsymbol W*\nu\right\|_{L^1(0,L)}
 \to 0.
\end{equation}
\end{proposition}

\begin{proof}
The dominated convergence theorem and the uniform local bound give the
first term of \eqref{eq:local-Riccati-Lp}.  The integral remainder in
Taylor's formula is
\[
 e^z-1-z-\frac{z^2}{2}
      =\frac{z^3}{2}\int_0^1(1-s)^2e^{sz}\dd s.
\]
Since $\operatorname{Re}\left(\boldsymbol V^{(n)}+\boldsymbol h\right)\le0$
coordinatewise, this formula implies
\begin{equation}\label{eq:Riccati-Taylor}
\begin{gathered}
 E_n:=b_n^2\mathcal W
          \left(\frac{\boldsymbol V^{(n)}+\boldsymbol h}{b_n}\right)
             -\frac12\left(\boldsymbol V^{(n)}+\boldsymbol h\right)^2,\\
 \|E_n\|_{L^\infty(0,L)}
   \le \frac{\|\boldsymbol V^{(n)}+\boldsymbol h\|_{L^\infty(0,L)}^3}
               {6b_n}
   \le\frac{C_L}{b_n}.
\end{gathered}
\end{equation}
Factoring the difference of the componentwise squares gives
\[
 \left\|\boldsymbol W^{(n)}-\boldsymbol W\right\|_{L^p(0,L)}
 \le L^{1/p}\|E_n\|_{L^\infty(0,L)}
        +C_L\left\|\boldsymbol V^{(n)}-\boldsymbol V\right\|_{L^p(0,L)} \to 0.
\]
Let $K_n(\dd s):=\boldsymbol R^{(n)}(s)\dd s$.
Its masses $m_{K_n}(L)$ are uniformly bounded by vague convergence and
positivity.  The identity in \eqref{eq:local-prelimit-Riccati},
\eqref{eq:local-measure-Young}, and
Lemma~\ref{lem:local-convolution-vague} yield
\begin{align}\label{eq:local-Riccati-L1-passage}
 \left\|\boldsymbol V^{(n)}-\boldsymbol W*\boldsymbol\Pi\right\|_{L^1(0,L)}
 &\le m_{K_n}(L)
          \left\|\boldsymbol W^{(n)}-\boldsymbol W\right\|_{L^1(0,L)}
       +\left\|\boldsymbol W*K_n-\boldsymbol W*\boldsymbol\Pi\right\|_{L^1(0,L)}
       \notag \to 0.
\end{align}
Combining this with \eqref{eq:local-Riccati-Lp} proves the limiting
Riccati equation almost everywhere.

Similarly, let
$m_{\nu_n}(L):=\sum_{i=1}^d\nu_{n,i}([0,L])$.
These masses are uniformly bounded, and
\begin{equation*}
    \left\|\boldsymbol W^{(n)}*\nu_n-\boldsymbol W*\nu\right\|_{L^1(0,L)}
   \le m_{\nu_n}(L)
         \left\|\boldsymbol W^{(n)}-\boldsymbol W\right\|_{L^1(0,L)}
       +\left\|\boldsymbol W*\nu_n-\boldsymbol W*\nu\right\|_{L^1(0,L)}
       \to 0
\end{equation*}  
by \eqref{eq:local-Riccati-Lp} and
Lemma~\ref{lem:local-convolution-vague}.  This proves
\eqref{eq:local-FLF-L1}.
\end{proof}

The uniqueness argument in \cite[Lemma~4.4]{Xu2024Scaling} also applies
to the bounded almost-everywhere solution class just obtained.
Indeed, let $D:=\boldsymbol\Pi(\{0\})$, and for two
$\mathbb C_-^d$-valued solutions let
$x:=\boldsymbol V+\boldsymbol h$,
$x^*:=\boldsymbol V^*+\boldsymbol h$, and
$\Delta:=\boldsymbol V-\boldsymbol V^*$.
The difference equation and the diagonal-atom condition give
\[
 \Delta_i(t)\left(1-\frac{D_{ii}}2(x_i(t)+x_i^*(t))\right)
 =\frac12\sum_{j=1}^d\int_{(0,t]}
       \left(x_j^2-(x_j^*)^2\right)(t-s)\Pi_{ji}(\dd s).
\]
The factor multiplying $\Delta_i$ has modulus at least one because
$\operatorname{Re}x_i,\operatorname{Re}x_i^*\le0$.
Local boundedness, factorization of the squares, and
\eqref{eq:local-measure-Young} consequently imply
\[
 \|e^{-\lambda\cdot}\Delta\|_{L^1(0,L)}
 \le C_L\left(\sum_{i=1}^d\sum_{j=1}^d
        \int_{(0,L]}e^{-\lambda s}\Pi_{ji}(\dd s)\right)
        \left\|e^{-\lambda\cdot}\Delta\right\|_{L^1(0,L)}.
\]
The coefficient tends to zero as $\lambda\to\infty$, so choosing it
less than one proves uniqueness almost everywhere.
For c\`adl\`ag representatives, this also proves equality at every time.

We now prove Theorem~\ref{thm:relaxed-Xu}.  The moment bounds give
stochastic compactness, and an auxiliary immigration rate of size
$\varepsilon$ fixes the logarithm branch needed to identify the
Riccati limit.

\begin{proof}[Proof of Theorem~\ref{thm:relaxed-Xu}]
Let $J_n(t):=\int_0^t\boldsymbol H^{(n)}(s)\dd s$.
The mean-intensity identity and Doob's $L^2$ maximal inequality give
\begin{equation}\label{eq:general-compactness-moments}
 \E\left[\boldsymbol{\mathcal A}^{(n)}(t)\right]
   =\E\left[\boldsymbol{\mathcal N}^{(n)}(t)\right]=J_n(t),
 \qquad
 \E\!\left[\sup_{0\le s\le L}
               \left|\boldsymbol{\mathcal M}^{(n)}(s)\right|_1^2\right]
       \le4d|J_n(L)|_1.
\end{equation}
The convergence in \eqref{eq:general-exogenous} and monotonicity bound
$J_n(L)$ uniformly, by first choosing a continuity point of
$\boldsymbol\Upsilon$ larger than $L$.
Thus the uniform-tightness and $S$-compactness arguments in
\cite[Proposition~5.3 and Lemma~5.4]{Xu2024Scaling} apply.
The counts and compensators have the same limit, since
\eqref{eq:count-compensator-difference} and
\eqref{eq:general-compactness-moments} imply
\[
 \E\!\left[\sup_{0\le s\le L}
    \left|\boldsymbol{\mathcal N}^{(n)}(s)
           -\boldsymbol{\mathcal A}^{(n)}(s)\right|_1^2\right]
       \le\frac{4d|J_n(L)|_1}{b_n^2}\to 0.
\]

For completeness, finite-dimensional identification at positive
continuity points of $\boldsymbol\Upsilon$ follows from time averages.
For $0\le s\le t$, the compensator mean and martingale isometry give
\begin{equation}\label{eq:general-mean-increments}
 \E\!\left[\left|\boldsymbol{\mathcal A}^{(n)}(t)
                      -\boldsymbol{\mathcal A}^{(n)}(s)\right|_1\right]
       =|J_n(t)-J_n(s)|_1,
 \qquad
 \E\!\left[\left|\boldsymbol{\mathcal M}^{(n)}(t)
                      -\boldsymbol{\mathcal M}^{(n)}(s)\right|_1^2\right]
       \le d|J_n(t)-J_n(s)|_1.
\end{equation}
For either process $Z^{(n)}$, let
$Z^{(n),\delta}(t):=\delta^{-1}\int_t^{t+\delta}Z^{(n)}(s)\dd s$.
Monotonicity for the compensator, and Jensen's inequality for the
martingale, bound their differences from these averages by
$|J_n(t+\delta)-J_n(t)|_1$ and $d|J_n(t+\delta)-J_n(t)|_1$,
in first and second mean, respectively.
The averages converge jointly along an $S$-convergent subsequence
\cite[Proposition~C.3(4)]{Xu2024Scaling}.  Choose $t+\delta$ to be a
continuity point of $\boldsymbol\Upsilon$, pass to the limit in $n$,
and then let $\delta\downarrow0$.  The limiting time averages converge
to the path values by right-continuity.  This proves joint
finite-dimensional convergence at every positive continuity point
$t$ of $\boldsymbol\Upsilon$.
Integration by parts,
\[
 \boldsymbol f*\dd\boldsymbol\Xi(t)
     =\boldsymbol f(0)\boldsymbol\Xi(t)
                       +\boldsymbol f'*\boldsymbol\Xi(t),
 \qquad
 \boldsymbol h*\dd\boldsymbol M(t)
     =\boldsymbol h(0)\boldsymbol M(t)
                       +\boldsymbol h'*\boldsymbol M(t),
\]
and its prelimit counterparts consequently give convergence of the
smooth-test transforms at these times.
The same integration-by-parts argument applies to smooth test functions of the
counts, because their difference from the compensators tends to zero
uniformly in probability.  These convergences hold jointly for any
finite collection of test functions and times.
If $\boldsymbol f(0)=\boldsymbol h(0)=\boldsymbol0$, the fixed-kernel
convolution consequence of $S$ convergence gives these transforms at
every time, without an evaluation of either path at that time.
The uniform second-moment bound also preserves the martingale property
\cite[Proposition~C.7]{Xu2024Scaling}.

Fix smooth test functions $\boldsymbol f,\boldsymbol h$ and an integer $L\ge1$.
The prelimit solutions exist and take values in $\mathbb C_-^d$ by
\cite[Lemma~4.3(1)]{Xu2024Scaling}, and depend on the kernel and test functions,
independently of the baseline.  Lemma~\ref{lem:uniform-Riccati-bound}
provides a constant $C_{L+1}$ such that
$\left|V_j^{(n)}(t)\right|\le C_{L+1}$ on $[0,L+1]$ for all large $n$ and all $j$.
Choose $\varepsilon>0$ with $\varepsilon C_{L+1}<\pi/4$.
For each $j=1,\ldots,d$, consider an auxiliary Hawkes array with the
same kernel and baseline
\begin{equation}\label{eq:auxiliary-small-immigration}
 \boldsymbol\mu_{n,j}^{\varepsilon}(s)
        :=\varepsilon b_n(\phi_{n,ij}(s))_{i=1}^d,
 \qquad
 \boldsymbol H_{j}^{(n),\varepsilon}(t)
        =\varepsilon\boldsymbol R^{(n)}_{\cdot j}(t).
\end{equation}
The second identity follows from the resolvent equation.
Its cumulative exogenous activity converges to
$\boldsymbol\Upsilon_j^{\varepsilon}(t)
 :=\varepsilon\boldsymbol\Pi_{\cdot j}([0,t])$
at continuity points, by Lemma~\ref{lem:positive-resolvent-limit}.
Applying the preceding compactness and time-average arguments to these
$d$ auxiliary arrays, and extracting a common subsequence, gives
limits $(\boldsymbol\Xi_j^{\varepsilon},\boldsymbol M_j^{\varepsilon})$.
The exact affine formula in \cite[Lemma~4.3(2)]{Xu2024Scaling} gives
\begin{equation}\label{eq:auxiliary-small-transform}
 \exp\left\{\varepsilon V_j^{(n)}(t)\right\}
 =\E\!\left[\exp\left\{
       \boldsymbol f*\dd\boldsymbol{\mathcal A}_{j}^{(n),\varepsilon}(t)
      +\boldsymbol h*\dd\boldsymbol{\mathcal M}_{j}^{(n),\varepsilon}(t)
                         \right\}\right]
 \longrightarrow F_j^{\varepsilon}(t)
\end{equation}
outside the countable discontinuity set of
$\boldsymbol\Upsilon_j^{\varepsilon}$, where
\[
 F_j^{\varepsilon}(t)
 :=\E\!\left[\exp\left\{
      \boldsymbol f*\dd\boldsymbol\Xi_j^{\varepsilon}(t)
      +\boldsymbol h*\dd\boldsymbol M_j^{\varepsilon}(t)\right\}\right].
\]
This function is c\`adl\`ag: the random exponent has c\`adl\`ag paths by
the integration-by-parts identities, and its exponential has modulus
at most one, so the dominated convergence theorem applies to both
one-sided limits.

On $[0,L+1]$, the prelimit values in
\eqref{eq:auxiliary-small-transform} lie in the closed sector
\[
 \left\{z\in\mathbb C:
       e^{-\varepsilon C_{L+1}}\le|z|\le1,
       \ |\arg z|\le\varepsilon C_{L+1}\right\}.
\]
Their limits lie in the same sector almost everywhere, and
right-continuity extends this property to $[0,L+1)$.
The principal logarithm is continuous on this sector and satisfies
\[
 V_j^{(n)}(t)
    =\varepsilon^{-1}\operatorname{Log}
                            e^{\varepsilon V_j^{(n)}(t)}
 \longrightarrow
 V_j(t):=\varepsilon^{-1}\operatorname{Log}F_j^{\varepsilon}(t)
 \quad\text{almost everywhere on }(0,L+1).
\]
The resulting $\boldsymbol V$ is c\`adl\`ag on $[0,L+1)$ and takes
values in $\mathbb C_-^d$.  Repeating the extraction for $L=1,2,\ldots$ gives a
common diagonal subsequence.  The local representatives agree almost
everywhere on overlapping intervals as limits of the same sequence,
and hence agree everywhere by right-continuity.
Proposition~\ref{prop:local-Riccati-repair} gives
$\boldsymbol V=\boldsymbol W*\boldsymbol\Pi$ almost everywhere.
Both sides are right-continuous, so the equation holds at every time.
The weighted uniqueness estimate above identifies the same global
solution for every extracted subsequence; the subsequence criterion
therefore gives full-sequence convergence in each finite local $L^p$ norm.

For an accumulation point $(\boldsymbol\Xi,\boldsymbol M)$ of the
original array, \eqref{eq:general-exogenous} gives
$\boldsymbol H^{(n)}(s)\dd s\to\dd\boldsymbol\Upsilon(s)$ vaguely.
Equation~\eqref{eq:local-FLF-L1} gives a further subsequence of exponents
converging almost everywhere to $\boldsymbol W*\dd\boldsymbol\Upsilon$.
Combining the exact affine formula with the smooth-test convergence
proved above establishes \eqref{eq:general-FLF} almost everywhere.
Both sides are right-continuous: the random exponential is bounded,
and the dominated convergence theorem applies to convolution of the
zero-extended c\`adl\`ag $\boldsymbol W$ against the locally finite
measure $\dd\boldsymbol\Upsilon$.  Thus \eqref{eq:general-FLF} holds
at every time.
Taking $\boldsymbol f(t)=\int_0^t g(s)\dd s$ and
$\boldsymbol h(t)=\int_0^t k(s)\dd s$ for continuous
$\mathbb C_-^d$- and $\mathrm i\R^d$-valued functions $g,k$ determines
the joint laws of integrals of the paths against continuous test functions.
Right-sided averages recover all path values, so these laws determine
the joint c\`adl\`ag law.  Uniqueness of $\boldsymbol V$ therefore
identifies all accumulation points and proves full-sequence $S$ convergence.

Finally, let $F_n(t):=\int_0^t\boldsymbol R^{(n)}(s)\dd s$ and
$F(t):=\boldsymbol\Pi([0,t])$.
Lemma~\ref{lem:positive-resolvent-limit}, positivity, and the dominated
convergence theorem give $F_n\to F$ in $L^1(0,L)$ for every $L$.
Integrating the prelimit martingale-resolvent identity
\cite[equations~(5.2) and~(5.13)]{Xu2024Scaling} yields
\begin{equation}\label{eq:integrated-general-Volterra}
 \int_0^t\boldsymbol{\mathcal A}^{(n)}(s)\dd s
       =\int_0^t J_n(s)\dd s
                      +F_n*\boldsymbol{\mathcal M}^{(n)}(t).
\end{equation}
Uniform $S$ tightness and the countable family of continuous functions
separating points give a further subsequence with an almost sure $S$
representation on each compact time interval
\cite[Theorem~1.1, Theorem~2.13(i), and Remark~2.16]{Jakubowski1997}.
The represented paths are uniformly bounded on that interval.
For each fixed $t$, the error bound
\[
 |F_n*\boldsymbol{\mathcal M}^{(n)}(t)-F*\boldsymbol M(t)|_1
 \le C_d\|F_n-F\|_{L^1(0,t)}
          \sup_{0\le s\le t}|\boldsymbol{\mathcal M}^{(n)}(s)|_1
      +|F*(\boldsymbol{\mathcal M}^{(n)}-\boldsymbol M)(t)|_1
 \longrightarrow0
\]
uses the convolution consequence of $S$ convergence
\cite[Proposition~C.1(2)]{Xu2024Scaling} for the fixed bounded kernel $F$.
Pass to the limit in \eqref{eq:integrated-general-Volterra}, first at
rational times and then at all times by continuity of the integrals.
Fubini's theorem gives
\[
 \int_0^t\boldsymbol\Xi(s)\dd s
      =\int_0^t\boldsymbol\Upsilon(s)\dd s+F*\boldsymbol M(t)
      =\int_0^t\left(\boldsymbol\Upsilon(s)
                                  +\boldsymbol\Pi*\boldsymbol M(s)\right)\dd s.
\]
Right-continuity of the integrands proves \eqref{eq:general-Volterra}.
Fatou's lemma applied to \eqref{eq:general-compactness-moments}, first
at positive continuity points of $\boldsymbol\Upsilon$ and then along
such points decreasing to any $t\ge0$, gives
\eqref{eq:general-limit-moments}.
Passing \eqref{eq:general-mean-increments} to the limit at continuity
points of $\boldsymbol\Upsilon$, and then approaching arbitrary
$s,t$ from the right, gives
\[
 \E[|\boldsymbol\Xi(t)-\boldsymbol\Xi(s)|_1]
       \le|\boldsymbol\Upsilon(t)-\boldsymbol\Upsilon(s)|_1,
 \qquad
 \E[|\boldsymbol M(t)-\boldsymbol M(s)|_1^2]
       \le d|\boldsymbol\Upsilon(t)-\boldsymbol\Upsilon(s)|_1.
\]
Thus continuity of $\boldsymbol\Upsilon$ implies stochastic continuity
of both processes.  If $\boldsymbol\Upsilon$ is continuous on $\R$,
its zero extension gives $\boldsymbol\Upsilon(0)=\boldsymbol0$, and
\eqref{eq:general-limit-moments} gives zero initial values.

It remains to verify the asserted continuity of $\boldsymbol V$.
If $\boldsymbol f(0)=\boldsymbol h(0)=\boldsymbol0$, the
integration-by-parts formulas make the auxiliary random exponents
continuous in time; dominated convergence makes
$F_j^\varepsilon$ continuous, and its fixed logarithm is continuous.
Alternatively, if $t\mapsto\boldsymbol\Pi([0,t])$ is continuous on
$\R_+$, the increment estimates just proved apply to the auxiliary
arrays and make their limits stochastically continuous at positive
times.  The same integration-by-parts formulas give convergence in
probability of their exponents as time varies.  Their bounded
exponentials therefore have continuous expectations there.
Right-continuity at zero, already established, proves continuity of
$\boldsymbol V$ on $\R_+$ in this case as well.
\end{proof}

A boundary layer explains why the integral norm is necessary, even with
smooth constant test functions and a diagonal potential atom.

\begin{remark} \label{rem:Riccati-boundary-layer}
The $L^\infty$ convergence asserted in
\cite[Lemma~4.5(2)]{Xu2024Scaling} can fail under its hypotheses.
In dimension one, for $n\ge2$, let
\[
 r_n=n^{-3/4},\quad \phi_n(s)=(1-r_n)e^{-s},\quad
 \theta_n=n^{1/2},\quad b_n=r_n^{-1},\quad
 \mu_n=(nr_n)^{-1},\qquad s>0.
\]
We set $\phi_n(0)=0$ by the kernel convention in
\eqref{eq:general-Hawkes}.
Then $\theta_n\to\infty$, and for every $\lambda\ge0$,
\[
 b_n\left[1-\mathcal L_{\phi_n}(\lambda/n)\right]
       =\frac{1+\lambda/(nr_n)}{1+\lambda/n}\to 1.
\]
Thus the limiting criticality function is $\Phi\equiv1$ and the
potential measure is $\Pi=\delta_0$.
The normalized resolvent is
$R^{(n)}(t)=nr_n(1-r_n)e^{-nr_nt}$, and
\[
 \int_0^tH^{(n)}(s)\dd s
   =t-\frac{1-r_n}{nr_n}(1-e^{-nr_nt})
   \to \Upsilon(t)=t.
\]
For $f\equiv-u$ with $u>0$ and $h\equiv0$,
\eqref{eq:local-prelimit-Riccati} becomes
\[
 V_n'(t)=nr_n F_n(V_n(t)),\quad V_n(0)=0,\qquad
 F_n(x)=(1-r_n)\left[-u+
          \frac{e^{r_nx}-1-r_nx}{r_n^2}\right]-x.
\]
On $(-\infty,0]$, $F_n'(x)\le-1$, while
$F_n(-u)>0>F_n(0)$, so $F_n$ has a unique root $v_n\in(-u,0)$.
Taylor's formula gives $v_n\to v:=1-\sqrt{1+2u}$, and the derivative
bound gives
\[
 0\le V_n(t)-v_n\le(-v_n)e^{-nr_nt}\le u e^{-nr_nt}.
\]
Consequently $V_n(t)\to v$ and $W_n(t)\to v$ for every $t>0$.
However, $W_n$ is continuous and $W_n(0)=-u$.  Every sufficiently
small interval to the right of zero therefore has positive measure
and values arbitrarily close to $-u$, giving
\[
 \|W_n-v\|_{L^\infty(0,L)}
       \ge|-u-v|=\frac{v^2}{2}>0,\qquad L>0.
\]
This example satisfies the almost-everywhere convergence premise and
the original high-intensity and diagonal-atom assumptions.
Proposition~\ref{prop:local-Riccati-repair} gives the finite-$p$
convergence and the convolution limits that remain valid.
\end{remark}

\subsection{A monotone-path \texorpdfstring{$M_1$}{M1} principle}

The next result is independent of the Hawkes structure.  Throughout, product $M_1$ means coordinatewise scalar $M_1$, on finite or infinite time horizons.  On $\R_+$, deterministic $M_1$ convergence is characterized by convergence of restrictions on $[0,L]$ at continuity points $L$ of the limit path \cite[Theorem~12.9.3]{Whitt2002}.  The corresponding weak-convergence criterion is joint convergence of the restricted vectors at deterministic $L$ where the limit is almost surely continuous; it follows by the continuous-mapping theorem and constant extension after $L$.

\begin{theorem} \label{thm:monotone-M1}
Let $\boldsymbol Z^n$ and $\boldsymbol Z$ be random elements of
$\cD(\R_+,\R^m)$ with nonnegative, coordinatewise nondecreasing paths.
Suppose there is a dense set $D\subset[0,\infty)$ containing $0$ such
that, for every $k\ge1$ and every $t_1,\ldots,t_k\in D$,
\[
 \left(\boldsymbol Z^n_{t_1},\ldots,\boldsymbol Z^n_{t_k}\right)
 \Longrightarrow
 (\boldsymbol Z_{t_1},\ldots,\boldsymbol Z_{t_k}).
\]
Then $\boldsymbol Z^n\Longrightarrow\boldsymbol Z$ in
$\cD(\R_+,\R^m)$ equipped with the product $M_1$ topology.
In particular, the restrictions converge in product $M_1$ on $[0,L]$
whenever $L>0$ satisfies
$\mathbb P\big(\Delta Z_L^j=0\text{ for all }j\big)=1$.
\end{theorem}

\begin{proof}
For each coordinate $j$, let
$Y_t^j:=Z_t^j/(1+Z_t^j)$ and $m_j(t):=\E[Y_t^j]$.
Bounded convergence theorem and the c\`adl\`ag property show that $m_j$ is
c\`adl\`ag; it is also nondecreasing and hence has at most countably
many discontinuities. At a time $t>0$ where $m_j$ has no jump,
\[
 0=m_j(t)-m_j(t-)
   =\E\!\left[Y_t^j-Y_{t-}^j\right].
\]
The difference inside the expectation is nonnegative, so it vanishes
almost surely. Since $z\mapsto z/(1+z)$ is strictly increasing,
$\Delta Z_t^j=0$ almost surely. Let
\[
 \mathcal T:=\{0\}\cup
 \left\{t>0:\mathbb P\left(\Delta Z_t^j=0\text{ for all }j\right)=1\right\}.
\]
The complement of $\mathcal T$ is countable, and $\boldsymbol Z$ is
stochastically continuous at every time in $\mathcal T$.

We extend the assumed finite-dimensional convergence from $D$ to
$\mathcal T$. Fix $t\in\mathcal T\setminus\{0\}$ and choose
$s_k,u_k\in D$ with $s_k\uparrow t$ and $u_k\downarrow t$, where
$s_k<t<u_k$. Coordinatewise monotonicity gives
\[
 \boldsymbol0\le\boldsymbol Z^n_t-\boldsymbol Z^n_{s_k}
       \le\boldsymbol Z^n_{u_k}-\boldsymbol Z^n_{s_k}.
\]
For each fixed $k$, joint convergence at $s_k,u_k$, the continuous-mapping
theorem, and Portmanteau theorem imply, for every $\varepsilon>0$,
\[
 \limsup_{n\to\infty}
 \mathbb P\!\left(\left\|\boldsymbol Z^n_t-\boldsymbol Z^n_{s_k}\right\|_\infty
                  >\varepsilon\right)
 \le
 \mathbb P\!\left(\left\|\boldsymbol Z_{u_k}-\boldsymbol Z_{s_k}\right\|_\infty
                  \ge\varepsilon\right).
\]
The right-hand side tends to zero as $k\to\infty$, because
$\boldsymbol Z$ is almost surely continuous at $t$; also
$\boldsymbol Z_{s_k}\to\boldsymbol Z_t$ in probability. Applying this
approximation simultaneously at finitely many times, and using $0\in D$
when a time is zero, proves that, for every $q\ge1$ and every
$t_1,\ldots,t_q\in\mathcal T$,
\begin{equation}\label{eq:monotone-fdd-continuity}
 \left(\boldsymbol Z^n_{t_1},\ldots,\boldsymbol Z^n_{t_q}\right)
 \Longrightarrow
 \left(\boldsymbol Z_{t_1},\ldots,\boldsymbol Z_{t_q}\right).
\end{equation}

Fix $L\in\mathcal T$ with $L>0$. For each coordinate $j$,
$Z_L^{n,j}\Longrightarrow Z_L^j$, so that
$\sup_{t\le L}Z_t^{n,j}=Z_L^{n,j}$ is tight. The interior scalar $M_1$
oscillation modulus is zero for every nondecreasing path. Choose
$\delta_k\downarrow0$ such that both $\delta_k$ and $L-\delta_k$
belong to $\mathcal T$. The finite-dimensional convergence
\eqref{eq:monotone-fdd-continuity}
and Portmanteau theorem give
\begin{align*}
 \limsup_{n\to\infty}
 \mathbb P\left(Z_{\delta_k}^{n,j}-Z_0^{n,j}>\varepsilon\right)
 &\le \mathbb P\left(Z_{\delta_k}^j-Z_0^j\ge\varepsilon\right),\\
 \limsup_{n\to\infty}
 \mathbb P\left(Z_L^{n,j}-Z_{L-\delta_k}^{n,j}>\varepsilon\right)
 &\le \mathbb P\left(Z_L^j-Z_{L-\delta_k}^j\ge\varepsilon\right).
\end{align*}
Both right-hand sides tend to zero as $k\to\infty$: the first by
right-continuity at $0$, and the second because $L\in\mathcal T$.
The endpoint oscillations are nondecreasing in $\delta$, so control along
$\delta_k\downarrow0$ gives the full endpoint conditions. The scalar
$M_1$ tightness criterion \cite[Theorem~12.12.3]{Whitt2002} now applies
coordinatewise, and finite $m$ gives tightness in product $M_1$.

Joint finite-dimensional convergence at the stochastic continuity times
in $\mathcal T\cap[0,L]$ identifies every subsequential limit as
$\boldsymbol Z|_{[0,L]}$. Thus the restrictions converge in product
$M_1$ on every such interval. Since $\mathcal T$ is unbounded and has
countable complement, these restriction limits give convergence in the
product $M_1$ topology on $\cD(\R_+,\R^m)$.
\end{proof}

We now apply Theorem~\ref{thm:monotone-M1} simultaneously to the rescaled compensator and count.  Continuity of the limiting cumulative exogenous activity provides the stochastic continuity needed for the upgrade and also permits a general $J_1$ obstruction.

\begin{corollary} \label{cor:general-M1}
Under the hypotheses of Theorem~\ref{thm:relaxed-Xu}, suppose $\boldsymbol\Upsilon$ is continuous on $\R$.  Then, on every compact interval, as $n\to\infty$,
\begin{equation}\label{eq:general-M1-limit}
    \left(\boldsymbol{\mathcal A}^{(n)},
          \boldsymbol{\mathcal N}^{(n)}\right)
    \Longrightarrow(\boldsymbol\Xi,\boldsymbol\Xi),
\end{equation}
in the product $M_1$ topology.  On any interval $[0,L]$ where a coordinate $\Xi_i$ has discontinuous paths with positive probability, convergence of either corresponding marginal in $J_1$ is impossible whenever $N_{n,i}$ is simple.
\end{corollary}

\begin{proof}
By Theorem~\ref{thm:relaxed-Xu}, the pair
$(\boldsymbol{\mathcal A}^{(n)},\boldsymbol{\mathcal N}^{(n)})$
has nonnegative, coordinatewise nondecreasing paths, and its
finite-dimensional distributions converge to those of
$(\boldsymbol\Xi,\boldsymbol\Xi)$ on a dense set
$D_0\subset(0,\infty)$ by the finite-dimensional assertion of
Theorem~\ref{thm:relaxed-Xu}. Both prelimit processes start at
$\boldsymbol0$. Moreover, continuity of $\boldsymbol\Upsilon$,
together with \eqref{eq:general-limit-moments}, gives
$\boldsymbol\Xi(0)=\boldsymbol0$ and stochastic continuity of
$\boldsymbol\Xi$. Hence the same joint finite-dimensional convergence
holds on $D:=D_0\cup\{0\}$. Applying Theorem~\ref{thm:monotone-M1}
to the $2d$-dimensional pair proves \eqref{eq:general-M1-limit}
on every $[0,L]$, since every $L>0$ is a stochastic continuity time
of the limit.

For the $J_1$ assertion, each $\mathcal A_i^{(n)}$ is continuous, whereas simplicity of $N_{n,i}$ gives
\[
    \sup_{0\le t\le L}
       \left|\Delta\mathcal N_i^{(n)}(t)\right|
       \le b_n^{-2}\longrightarrow0.
\]
If either marginal converged in $J_1$, the Skorokhod representation theorem would allow an almost surely $J_1$-convergent realization.  The jump-matching property of $J_1$ then implies that every jump of the limit must be approximated by jumps of the prelimit paths.  Continuous approximants, or approximants whose maximal jump tends to zero, can therefore have only a continuous $J_1$ limit.  This contradicts
\[
    \mathbb P\left(\Xi_i\text{ is discontinuous on }[0,L]\right)>0.
\]
The proof is complete.
\end{proof}

\subsection{Diagonal multitype IG}

For the array in Section~\ref{sec:model}, the intrinsic activity scale is
$b_T=r_T^{-1}\to\infty$, independently of the behavior of $Tr_T^2$.  The limiting criticality data are
\begin{equation}\label{eq:Xu-multitype-data}
    \boldsymbol K=\boldsymbol I,
    \qquad \boldsymbol\Phi=\boldsymbol B,
    \qquad \boldsymbol\Pi=\boldsymbol B^{-1}\delta_0,
    \qquad
    \boldsymbol\Upsilon(t)=t\boldsymbol B^{-1}\boldsymbol\eta.
\end{equation}
The atom $\boldsymbol\Pi(\{0\})=\boldsymbol B^{-1}$ satisfies the diagonal-atom assumption of \cite{Xu2024Scaling} precisely in the diagonal specialization considered below.

\begin{proposition} \label{prop:general-diagonal-IG}
Assume \eqref{eq:multivariate-scaling} and suppose
$\boldsymbol B=\operatorname{diag}(b_1,\ldots,b_d)$.  Then
Theorem~\ref{thm:relaxed-Xu} and Corollary~\ref{cor:general-M1} apply with
\begin{equation}\label{eq:Xu-diagonal-specialization}
    n=T,\qquad b_T=r_T^{-1},\qquad
    \theta_T=\frac1{Tr_T^2},
\end{equation}
and the data in \eqref{eq:Xu-multitype-data}.  They yield the diagonal case of Theorem~\ref{thm:main}, whose limit is the vector of independent IG subordinators in Corollary~\ref{cor:diagonal-IG}.
\end{proposition}

\begin{proof}
First, $T\theta_T=r_T^{-2}\to\infty$.  For every $\lambda\ge0$, the inequality $1-e^{-x}\le x$ gives the entrywise bound
\[
 \boldsymbol0\le\boldsymbol A_T
       -\mathcal L_{\boldsymbol\phi_T}(\lambda/T)
 \le \frac{\lambda}{T}\boldsymbol D_T.
\]
Consequently,
\[
 b_T\left[\boldsymbol I
      -\mathcal L_{\boldsymbol\phi_T}(\lambda/T)\right]
 =\boldsymbol B_T+o(1)\longrightarrow\boldsymbol B,
\]
because $Tr_T\to\infty$.  Thus $\boldsymbol K=\boldsymbol I$ and
$\boldsymbol\Phi=\boldsymbol B$.  The common tilted-stability
condition holds with $\lambda_0=0$, since
$\rho(\mathcal L_{\boldsymbol\phi_T}(0))=\rho(\boldsymbol A_T)<1$.
In the notation of
\eqref{eq:general-potential}, $\ell=d$, so the potential measure is
$\boldsymbol\Pi=\boldsymbol B^{-1}\delta_0$, whose atom is diagonal under the present hypothesis.

Moreover,
\[
    \boldsymbol\mu^{(T)}(t)
      =Tr_T\boldsymbol\mu_T\longrightarrow\boldsymbol\eta.
\]
Resolvent convergence and \eqref{eq:general-rescaling} therefore give
$\mathcal I_{\boldsymbol H^{(T)}}(t)\to
t\boldsymbol B^{-1}\boldsymbol\eta$, proving condition~\eqref{eq:general-exogenous} on cumulative exogenous activity.

To identify the limit, take $\boldsymbol h=\boldsymbol0$.
For a real nonpositive smooth test $\boldsymbol f$, the atom
$\boldsymbol\Pi=\boldsymbol B^{-1}\delta_0$ reduces
\eqref{eq:general-Riccati} to
\[
 b_iV_i=f_i+\frac12V_i^2,
 \qquad
 V_i=b_i-\sqrt{b_i^2-2f_i},
 \qquad W_i=b_iV_i.
\]
To identify the increments, fix $0=t_0<t_1<\cdots<t_m$ and
$\boldsymbol u_1,\ldots,\boldsymbol u_m\in\R_+^d$.
Approximate the negative step test taking value $-\boldsymbol u_k$
on $[t_m-t_k,t_m-t_{k-1})$ by uniformly bounded
$\mathbb R_-^d$-valued $C^1$ test functions.
The stochastic continuity of $\boldsymbol\Xi$ implies that
$\dd\boldsymbol\Xi$ has no atom at any of the finitely many
deterministic endpoints, almost surely.
Dominated convergence therefore applies to the random Stieltjes
integrals and their bounded exponentials.
On the deterministic side the algebraic Riccati root is
$V_i=b_i-\sqrt{b_i^2-2f_i}$, and
$\dd\boldsymbol\Upsilon(s)=\boldsymbol B^{-1}\boldsymbol\eta\dd s$,
so dominated convergence applies there as well.  Consequently,
\[
 \E\left[\exp\left\{-\sum_{k=1}^m
    \langle\boldsymbol u_k,
           \boldsymbol\Xi(t_k)-\boldsymbol\Xi(t_{k-1})\rangle
                    \right\}\right]
 =\exp\left\{-\sum_{k=1}^m(t_k-t_{k-1})\sum_{i=1}^d
       \eta_i\left(\sqrt{b_i^2+2u_{k,i}}-b_i\right)\right\}.
\]
The factorization identifies stationary independent increments and
independence of the coordinates.  Corollary~\ref{cor:general-M1}
now gives the joint product-$M_1$ convergence.
\end{proof}

\begin{remark}\label{rem:non-diagonal-Xu}
If $\boldsymbol B$ is non-diagonal, the formal specialization still gives the atom $\boldsymbol B^{-1}\delta_0$, but
$\boldsymbol B^{-1}$ is non-diagonal.  It therefore violates the diagonal-atom condition \eqref{eq:general-diagonal-atom} \cite[Assumption~2.8]{Xu2024Scaling};
\cite[Remark~2.21]{Xu2024Scaling} retains relative compactness in this situation but not the uniqueness conclusion needed to identify the entire sequence.  Theorem~\ref{thm:main} does not invoke that extension: its cluster proof identifies the exponent directly, and the nonsingular $M$-matrix estimate in Lemma~\ref{lem:multitype-family} provides uniqueness.
\end{remark}

\begin{remark} \label{rem:oscillating}
The accumulated-activity formulation also covers arrays whose original microscopic intensity scale oscillates.  For example, along integers let $r_n=n^{-3/4}$ for even $n$ and $r_n=n^{-1/4}$ for odd $n$, and take $\mu_n=\vartheta/(nr_n)$.  Then $r_n\to0$, $nr_n\to\infty$, and $n\mu_nr_n=\vartheta$, but $nr_n^2$ alternates between a sequence tending to zero and one tending to infinity.  Accordingly, $\theta_n=(nr_n^2)^{-1}$ has no limit, whereas $n\theta_n=r_n^{-2}\to\infty$, so the scalar case of Proposition~\ref{prop:general-diagonal-IG} applies without passing to subsequences.
\end{remark}

\section{Multitype cluster proof}\label{sec:multiproof}

We prove Theorem~\ref{thm:main} directly from the multitype
immigration--branching representation.  A type-$j$ event produces, independently for $i=1,\ldots,d$, a Poisson process of type-$i$ children with intensity
$\phi_{T,ij}(s)\dd s$.  Thus its mean offspring vector is the $j$-th column of $\boldsymbol A_T$.  Since $\rho(\boldsymbol A_T)<1$, every immigrant family is finite almost surely.

For a family rooted at one type-$j$ immigrant, let
$\boldsymbol S^{T,j}\in\mathbb N^d$ be its vector of total progeny, including the immigrant, and define
\begin{equation}\label{eq:multitype-family-transform}
    q_{T,j}(\boldsymbol u)
      :=\E\left[e^{-r_T^2
         \left\langle\boldsymbol u,\boldsymbol S^{T,j}\right\rangle}\right],
    \qquad
    \boldsymbol z_T(\boldsymbol u)
      :=\frac{\boldsymbol1-\boldsymbol q_T(\boldsymbol u)}{r_T}.
\end{equation}

The first step is to identify the asymptotic family mass generated by an immigrant of each type.  The following lemma shows that the scaled defects of the joint family transforms converge to precisely the coupled quadratic root defining the multitype IG exponent.

\begin{lemma} \label{lem:multitype-family}
Assume $r_T\to0$, $\rho(\boldsymbol A_T)<1$, and
$\boldsymbol B_T=(\boldsymbol I-\boldsymbol A_T)/r_T\to\boldsymbol B$,
where $\boldsymbol B$ is a nonsingular $M$-matrix.
For every $\boldsymbol u\in\R_+^d$,
\begin{equation}\label{eq:multitype-family-limit}
    \boldsymbol z_T(\boldsymbol u)
      \to\boldsymbol z(\boldsymbol u),
\end{equation}
where $\boldsymbol z(\boldsymbol u)$ is the unique nonnegative solution of
\eqref{eq:multitype-IG-root}.  The convergence is locally uniform on
$\R_+^d$.
\end{lemma}

\begin{proof}
The branching property and independence of the Poisson offspring counts give
\begin{equation}\label{eq:multitype-pgf}
    q_{T,j}(\boldsymbol u)
      =\exp\left\{-r_T^2u_j
          +\sum_{i=1}^d A_{T,ij}
             (q_{T,i}(\boldsymbol u)-1)\right\}.
\end{equation}
Taking logarithms, substituting
$\boldsymbol q_T=\boldsymbol1-r_T\boldsymbol z_T$, and dividing by
$r_T^2$ yields the exact vector identity
\begin{equation}\label{eq:multitype-exact-root}
 \boldsymbol u
  =\boldsymbol B_T^{\top}\boldsymbol z_T
    +\boldsymbol g_T(\boldsymbol z_T),
 \qquad
 g_{T,i}(z):=\frac{-\log(1-r_Tz)-r_Tz}{r_T^2}.
\end{equation}

Choose $\boldsymbol p>\boldsymbol0$ such that
$\boldsymbol B\boldsymbol p>\boldsymbol0$, which is possible for a nonsingular $M$-matrix.  For all sufficiently large $T$,
$\boldsymbol B_T\boldsymbol p\ge c\boldsymbol1$ for some $c>0$.
Taking the inner product of \eqref{eq:multitype-exact-root} with
$\boldsymbol p$ and using $g_{T,i}\ge0$ gives
\[
    \langle\boldsymbol p,\boldsymbol u\rangle
      \ge c\|\boldsymbol z_T(\boldsymbol u)\|_1.
\]
Thus $\boldsymbol z_T$ is locally bounded.  Along any convergent subsequence,
$r_T\boldsymbol z_T\to\boldsymbol0$ and
$g_{T,i}(z_{T,i})\to z_i^2/2$, so every subsequential limit solves
\eqref{eq:multitype-IG-root}.  If $\boldsymbol z$ and
$\widetilde{\boldsymbol z}$ are two nonnegative solutions, then
\[
 \left[\boldsymbol B^{\top}
   +\frac12\operatorname{diag}\left(\boldsymbol z+\widetilde{\boldsymbol z}\right)\right]
 \left(\boldsymbol z-\widetilde{\boldsymbol z}\right)=\boldsymbol0.
\]
The bracket is a nonsingular $M$-matrix, hence invertible.  Thus the
nonnegative solution is unique, and every convergent subsequence of
$\boldsymbol z_T(\boldsymbol u)$ has the same limit
$\boldsymbol z(\boldsymbol u)$.  Together with the local boundedness
proved above, this establishes the pointwise convergence
\eqref{eq:multitype-family-limit}.

We next verify continuity of $\boldsymbol z$.  The bound obtained by
taking the inner product with $\boldsymbol p$ also holds for the limiting
equation, so $\boldsymbol z(\boldsymbol u_k)$ is bounded whenever
$\boldsymbol u_k\to\boldsymbol u$.  Every convergent subsequence then
solves \eqref{eq:multitype-IG-root} at $\boldsymbol u$; uniqueness implies
$\boldsymbol z(\boldsymbol u_k)\to\boldsymbol z(\boldsymbol u)$.

To upgrade pointwise convergence to local uniform convergence, fix
$M>0$ and $\varepsilon>0$.  Uniform continuity of $\boldsymbol z$ on
$[0,M]^d$ permits a finite rectangular grid with vertex set $G$ such
that the lower and upper vertices $\boldsymbol v,\boldsymbol w$ of each
cell satisfy
$\|\boldsymbol z(\boldsymbol w)-\boldsymbol z(\boldsymbol v)\|_\infty
<\varepsilon$.  Since $G$ is finite, pointwise convergence gives
\[
 \max_{\boldsymbol a\in G}
 \|\boldsymbol z_T(\boldsymbol a)-\boldsymbol z(\boldsymbol a)\|_\infty
 <\varepsilon
\]
for all sufficiently large $T$.  Each component of $\boldsymbol z_T$ is
nondecreasing in every component of $\boldsymbol u$, as follows from
\eqref{eq:multitype-family-transform}; the same property holds for its
pointwise limit $\boldsymbol z$.  For a point $\boldsymbol u$ in a cell
with $\boldsymbol v\le\boldsymbol u\le\boldsymbol w$, we therefore have,
componentwise,
\[
 \boldsymbol z_T(\boldsymbol v)\le\boldsymbol z_T(\boldsymbol u)
       \le\boldsymbol z_T(\boldsymbol w),
 \qquad
 \boldsymbol z(\boldsymbol v)\le\boldsymbol z(\boldsymbol u)
       \le\boldsymbol z(\boldsymbol w).
\]
The grid-point error and the oscillation of $\boldsymbol z$ on each cell
consequently give
\[
 \sup_{\boldsymbol u\in[0,M]^d}
 \|\boldsymbol z_T(\boldsymbol u)-\boldsymbol z(\boldsymbol u)\|_\infty
 \le2\varepsilon
\]
for all sufficiently large $T$.  Every compact subset of $\R_+^d$ is
contained in such a cube, proving local uniform convergence.
\end{proof}

Let $(\tau_{j,k}^T)_{k\ge1}$ be the type-$j$ immigrant times; these are independent Poisson processes of rates $\mu_{T,j}$.  Attach independent family vectors
$\boldsymbol S_{j,k}^T$ with law $\boldsymbol S^{T,j}$ and collapse every family at its immigrant time:
\begin{equation}\label{eq:multitype-collapsed}
    \overline{\boldsymbol X}_t^T
      :=r_T^2\sum_{j=1}^d
        \sum_{k:\tau_{j,k}^T\le Tt}\boldsymbol S_{j,k}^T.
\end{equation}

The one-family asymptotics now determine the L\'evy exponent of the collapsed process.  Combined with the immigration balance, they yield convergence of the full compound-Poisson family process to the multitype IG subordinator in the following lemma.

\begin{lemma} \label{lem:multitype-collapsed}
Assume the hypotheses of Lemma~\ref{lem:multitype-family} and the
immigration balance $Tr_T\boldsymbol\mu_T\to\boldsymbol\eta$, where
$\boldsymbol\eta\in(0,\infty)^d$.
For every $L>0$,
\begin{equation}\label{eq:multitype-collapsed-limit}
    \overline{\boldsymbol X}^T\Longrightarrow\boldsymbol Z
\end{equation}
in $\cD([0,L],\R^d)$ with the $J_1$ topology, and hence also with product $M_1$-topology.
\end{lemma}

\begin{proof}
The process in \eqref{eq:multitype-collapsed} is a multivariate compound-Poisson subordinator.  Its Laplace exponent per unit macroscopic time is
\begin{align*}
 \psi_T(\boldsymbol u)
 =T\sum_{j=1}^d\mu_{T,j}
       (1-q_{T,j}(\boldsymbol u))
       =\left\langle Tr_T\boldsymbol\mu_T,
                    \boldsymbol z_T(\boldsymbol u)\right\rangle
   \longrightarrow
   \langle\boldsymbol\eta,\boldsymbol z(\boldsymbol u)\rangle.
\end{align*}
The limiting exponent is continuous at the origin.  The extended continuity theorem for multivariate Laplace transforms \cite[Theorem~6.23]{Kallenberg2021}, followed by the L\'evy-process convergence criterion \cite[Corollary~VII.3.6]{JacodShiryaev2003}, therefore gives \eqref{eq:multitype-collapsed-limit}.  This argument also proves that
\eqref{eq:multitype-IG-LT} defines a multivariate subordinator.
For the vector $\boldsymbol p>\boldsymbol0$ used in
Lemma~\ref{lem:multitype-family}, equation
\eqref{eq:multitype-IG-root} gives, for $c\ge1$,
\[
 \frac12\sum_{i=1}^d p_i z_i(c\boldsymbol u)^2
 \le c\langle\boldsymbol p,\boldsymbol u\rangle.
\]
If $\boldsymbol\delta\ge\boldsymbol0$ denotes the drift vector in the
L\'evy--Khintchine formula, its nonnegative integral term gives
\[
 0\le\langle\boldsymbol\delta,\boldsymbol u\rangle
 \le c^{-1}\langle\boldsymbol\eta,
                       \boldsymbol z(c\boldsymbol u)\rangle
 \longrightarrow0.
\]
Taking $\boldsymbol u=\boldsymbol1$ proves
$\boldsymbol\delta=\boldsymbol0$, so the subordinator is pure jump.
\end{proof}

It remains to restore the calendar times of all descendants.  Let
\begin{equation}\label{eq:matrix-renewal-measure}
    \boldsymbol U_T(\dd s)
      :=\boldsymbol I\delta_0(\dd s)
        +\sum_{n=1}^\infty\boldsymbol\phi_T^{*n}(s)\dd s
\end{equation}
be the matrix expected-age measure of a family.  Matrix convolution and Fubini's theorem give
\begin{equation}\label{eq:matrix-age-moment}
 \int_0^\infty s\boldsymbol U_T(\dd s)
   =(\boldsymbol I-\boldsymbol A_T)^{-1}
      \boldsymbol D_T
     (\boldsymbol I-\boldsymbol A_T)^{-1}.
\end{equation}

The matrix age identity \eqref{eq:matrix-age-moment} quantifies the expected family mass that can cross a fixed calendar-time boundary.  It yields both the vanishing error needed to undo the artificial family collapse and the increment bound required for product-$M_1$ tightness in the following lemma.

\begin{lemma} \label{lem:multitype-collapse}
For every fixed $t\ge0$, componentwise,
\begin{equation}\label{eq:multitype-collapse-error}
 \boldsymbol0\le
 \E\left[\overline{\boldsymbol X}_t^T-\boldsymbol X_t^T\right]
 \le
 \boldsymbol B_T^{-1}\boldsymbol D_T
       \boldsymbol B_T^{-1}\boldsymbol\mu_T
 \longrightarrow\boldsymbol0.
\end{equation}
Moreover, for $0\le s\le t\le L$,
\begin{equation}\label{eq:multitype-increment-bound}
 \E\left[\boldsymbol X_t^T-\boldsymbol X_s^T\right]
 \le (t-s)Tr_T\boldsymbol B_T^{-1}\boldsymbol\mu_T.
\end{equation}
Consequently, $(\boldsymbol X^T)$ is asymptotically tight as
$T\to\infty$ in product $M_1$ on every compact interval.
\end{lemma}

\begin{proof}
The difference in \eqref{eq:multitype-collapse-error} consists of descendants whose immigrant arrived before $Tt$ but whose own birth occurs after $Tt$.  For any nonnegative locally finite measure $\nu$ on $[0,\infty)$, Tonelli's theorem gives the tail integral identity
\begin{equation}\label{eq:tail-integral-identity}
 \int_0^\infty\nu((v,\infty))\dd v
 =\int_{[0,\infty)}
      \left(\int_0^\infty\1_{\{v<s\}}\dd v\right)\nu(\dd s)
 =\int_{[0,\infty)}s\nu(\dd s).
\end{equation}
Applying Campbell's formula to each type's immigrant Poisson process
\cite[Proposition~2.7]{LastPenrose2017}, and then applying
\eqref{eq:tail-integral-identity} entrywise to $\boldsymbol U_T$, gives
\begin{align*}
 \E\left[\overline{\boldsymbol X}_t^T-\boldsymbol X_t^T\right]
 &=r_T^2\int_0^{Tt}
       \boldsymbol U_T((Tt-s,\infty))\boldsymbol\mu_T\dd s\\
 &\le r_T^2\int_0^\infty
       \boldsymbol U_T((v,\infty))\boldsymbol\mu_T\dd v
       =r_T^2\left(\int_0^\infty s\boldsymbol U_T(\dd s)\right)
       \boldsymbol\mu_T.
\end{align*}
Using \eqref{eq:matrix-age-moment} and
$(\boldsymbol I-\boldsymbol A_T)^{-1}=r_T^{-1}\boldsymbol B_T^{-1}$
yields the upper bound in \eqref{eq:multitype-collapse-error}.  The matrices
$\boldsymbol B_T^{-1}$ are uniformly bounded, $\boldsymbol D_T$ is bounded by assumption, and
$\boldsymbol\mu_T=O((Tr_T)^{-1})$, so the bound vanishes.

The mean intensity of the process started from empty history is bounded componentwise by
$(\boldsymbol I-\boldsymbol A_T)^{-1}\boldsymbol\mu_T$.  Integrating over $(Ts,Tt]$ and multiplying by $r_T^2$ gives
\eqref{eq:multitype-increment-bound}.  Its right-hand side converges to
$(t-s)\boldsymbol B^{-1}\boldsymbol\eta$.  Each coordinate of
$\boldsymbol X^T$ is nonnegative and nondecreasing, so compact containment and the two endpoint controls follow from this bound and Markov's inequality.  The scalar monotone-path $M_1$ criterion applied coordinatewise proves product-$M_1$ tightness.
\end{proof}

\begin{proof}[Proof of Theorem~\ref{thm:main}]
Lemma~\ref{lem:multitype-collapsed} and the fixed-time estimate
\eqref{eq:multitype-collapse-error} imply convergence of all finite-dimensional distributions of $\boldsymbol X^T$ to those of
$\boldsymbol Z$.  Lemma~\ref{lem:multitype-collapse} gives product-$M_1$
tightness.  Since a L\'evy process is stochastically continuous, the finite-dimensional laws identify every subsequential limit, and hence
$\boldsymbol X^T\Rightarrow\boldsymbol Z$ in product $M_1$.

For each coordinate, Doob's inequality and
\eqref{eq:multitype-increment-bound} yield
\[
 \E\left[\sup_{0\le t\le L}
   \left|r_T^2M_{i,Tt}^T\right|^2\right]
 \le4r_T^4\E\left[N_{i,TL}^T\right]=O(r_T^2)\longrightarrow0.
\]
Thus $\boldsymbol X^T-\boldsymbol C^T\to\boldsymbol0$ uniformly in probability, and Slutsky's theorem gives \eqref{eq:main-convergence}.

Finally, the type-$j$ immigrant intensity measure
$\mu_{T,j}\dd s$ and every offspring intensity measure
$\phi_{T,ij}(s)\dd s$ are atomless: each assigns zero mass to every
singleton $\{s\}$.  The Poisson cluster construction therefore produces
no simultaneous births almost surely, so every coordinate count is
simple, meaning that each of its jumps has size one.  Its normalized jumps have size
$r_T^2\to0$, while every compensator path is continuous.  The $J_1$
jump-matching property therefore rules out convergence to any coordinate
$Z_i$ that is discontinuous with positive probability.
\end{proof}

\section{Scalar families and exact calendar collapse}\label{sec:collapsed}

The multitype cluster estimates specialize to explicit scalar family
laws.  We record these laws and use the expected-age distribution to
identify the exact calendar-collapse condition for a fixed delay density.

\subsection{Scalar family laws}

An immigrant family is a Poisson Galton--Watson tree with offspring
law $\Pois(a_T)$.  Let $S_T$ be its total progeny, including the
immigrant.  Since $a_T<1$, the family is finite almost surely and its
mean size is $\E[S_T]=\sum_{n=0}^\infty a_T^n=r_T^{-1}$.

If $\tau_i^T$ are the immigrant times and $S_{T,i}$ are independent copies of $S_T$, define the collapsed-cluster process
\begin{equation}\label{eq:collapsed}
    \overline X_t^T
    :=r_T^2\sum_{i:\,\tau_i^T\le Tt}S_{T,i}.
\end{equation}
Unlike the original Hawkes count, $\overline X^T$ places every descendant at the immigrant time.  It is a compound Poisson subordinator with macroscopic arrival rate $T\mu_T$ and jump law $r_T^2S_T$.

Let $\mathcal T_{T,j}$ denote the complete genealogical tree rooted at the $j$-th immigrant, so that $|\mathcal T_{T,j}|=S_{T,j}$.  The immigrant forest consists of the disjoint union of the complete trees whose roots arrive by time $Tt$:
\begin{equation}\label{eq:immigrant-forest}
    K_T(t):=\#\left\{j:\tau_j^T\le Tt\right\},
    \qquad
    \mathcal F_T(t):=\bigsqcup_{j=1}^{K_T(t)}\mathcal T_{T,j},
    \qquad
    |\mathcal F_T(t)|=\sum_{j=1}^{K_T(t)}S_{T,j}.
\end{equation}
Conditional on $K_T(t)$, these trees are independent Poisson Galton--Watson trees with offspring law $\Pois(a_T)$.  Each tree includes all descendants of its root, including those born after $Tt$; the forest is empty when $K_T(t)=0$.  Thus $\overline X_t^T=r_T^2|\mathcal F_T(t)|$, whereas $X_t^T$ counts only individuals already born by $Tt$.

Under $r_T\to0$ and
$\vartheta_T=T\mu_Tr_T\to\vartheta\in(0,\infty)$, the process in
\eqref{eq:collapsed} is the scalar specialization of
\eqref{eq:multitype-collapsed}, with $B=1$ and $\eta=\vartheta$.
Lemma~\ref{lem:multitype-collapsed} therefore gives, for every $L>0$,
\begin{equation}\label{eq:scalar-collapsed-limit}
    \overline X^T\Longrightarrow Y
    \quad\text{in }\cD([0,L],\R)\text{ with }J_1,
\end{equation}
where $Y$ is defined by \eqref{eq:IG-LT}.

\subsection{The exact calendar-collapse condition}

Lemma~\ref{lem:multitype-collapse} gives sufficiency of
$Tr_T\to\infty$.  For a fixed scalar delay density, the expected-age
distribution also proves necessity.

Let $F(\dd s):=\phi(s)\dd s$ and define the expected family age measure
\begin{equation}\label{eq:renewal-measure}
    U_T(\dd s):=\sum_{n=0}^\infty a_T^nF^{*n}(\dd s).
\end{equation}
The atom at zero is the immigrant, and the $n$-th term is the mean age distribution of generation $n$.  By \eqref{eq:kernel-assumption},
\begin{equation}\label{eq:age-moment}
    \int_0^\infty sU_T(\dd s)
    =m_\phi\sum_{n=1}^\infty na_T^n
    =\frac{m_\phi a_T}{(1-a_T)^2}.
\end{equation}

The scalar specialization of \eqref{eq:multitype-collapse-error},
with $B_T=1$ and $D_T=a_Tm_\phi$, gives, for every $t\ge0$,
\begin{equation}\label{eq:collapse-error}
 0\le\E[\overline X_t^T-X_t^T]\le\mu_Tm_\phi a_T.
\end{equation}
The bound uses only subcritical reproduction and its finite mean delay.

The following lemma characterizes calendar collapse under the
rare-family immigration balance.

\begin{lemma} 
\label{lem:collapse-necessity}
Let $\phi$ be a fixed probability density on $(0,\infty)$ with
$m_\phi\in(0,\infty)$.  Suppose $r_T\to0$ and
$\vartheta_T=T\mu_Tr_T\to\vartheta\in(0,\infty)$ as $T\to\infty$.
For every fixed $t>0$, the following three conditions are equivalent:
\[
 \overline X_t^T-X_t^T\xrightarrow{\mathbb P}0,
 \quad
 \E[\overline X_t^T-X_t^T]\to 0
 \quad \mbox{and}\quad 
 Tr_T\to \infty.
\]
\end{lemma}

\begin{proof}
Let $G_T$ have the geometric law
$\mathbb P(G_T=k)=r_Ta_T^k$, $k\ge0$, and let
$(L_j)_{j\ge1}$ be independent delays with density $\phi$,
independent of $G_T$.
The age $D_T:=\sum_{j=1}^{G_T}L_j$ has probability law $r_TU_T$.
Finite mean gives
$\widehat\phi(qr_T)=1-qm_\phi r_T+o(r_T)$ for every $q\ge0$.
Hence
\[
 \E[e^{-qr_TD_T}]
   =\frac{r_T}{1-a_T\widehat\phi(qr_T)}
   \longrightarrow\frac1{1+qm_\phi}.
\]
The continuity theorem for Laplace transforms
\cite[Theorem~6.23]{Kallenberg2021} gives
$r_TD_T\Rightarrow m_\phi E$, where $E$ is exponential with mean one.
Campbell's formula now yields
\begin{equation}\label{eq:collapse-necessity-mean}
 \E[X_t^T]
 =\mu_Tr_T^2\int_0^{Tt}U_T([0,Tt-s])\dd s
 =\vartheta_T\int_0^t
       \mathbb P(r_TD_T\le Tr_Tw)\dd w.
\end{equation}
If $Tr_T\to c\in[0,\infty)$ along a subsequence, continuity of the
limiting exponential distribution and dominated convergence imply
\[
 \E[X_t^T]\longrightarrow\vartheta\ell_c(t),
 \qquad
 \ell_c(t):=\int_0^t(1-e^{-cw/m_\phi})\dd w<t.
\]
This also covers $c=0$, for which $\ell_0(t)=0$.
Since $\E[\overline X_t^T]=\vartheta_Tt$, the expected collapse
error tends to $\vartheta(t-\ell_c(t))>0$ on that subsequence.
Thus convergence of the error in mean forces $Tr_T\to\infty$.

Equation~\eqref{eq:scalar-collapsed-limit} gives
$\overline X_t^T\Rightarrow Y_t$ under the present assumptions.
Moreover, $\E[\overline X_t^T]\to\E[Y_t]=\vartheta t$.
For $K>0$, convergence of expectations of the bounded continuous
function $x\mapsto x\wedge K$ gives
\[
 \lim_{T\to\infty}\E[(\overline X_t^T-K)^+]
       =\E[(Y_t-K)^+]\longrightarrow0
       \quad\text{as }K\to\infty.
\]
Therefore the nonnegative variables $\overline X_t^T$ are
asymptotically uniformly integrable.
The bound $0\le\overline X_t^T-X_t^T\le\overline X_t^T$ shows that
convergence of the error in probability implies convergence in mean.
Finally, if $Tr_T\to\infty$, the estimate in
equation~\eqref{eq:collapse-error} gives
$\E[\overline X_t^T-X_t^T]\le m_\phi a_T\vartheta_T/(Tr_T)\to0$,
and Markov's inequality gives convergence in probability.
\end{proof}

\section{Branching structure and universality}\label{sec:branching}

This section complements the Volterra first-passage interpretation in \cite[Lemmas~2.3--2.4 and Theorem~3.1]{AbiJaberAttalRosenbaum2025} with a branching construction.  Proposition~\ref{prop:forest-passage} makes explicit, for the present near-critical immigrant forest, the classical random-walk-to-Brownian first-passage mechanism.  Proposition~\ref{prop:universality} then extends the conclusion beyond Poisson offspring and shows that the IG class persists for finite-variance offspring laws in age-dependent branching clusters, with the offspring variance changing the dispersion parameter but not the limiting mean.  In terms of financial applications, this identifies a robustness property of the limiting activity clock: its mean activity rate is fixed by the immigration balance, while microscopic offspring dispersion controls the variability and tail scale of macroscopic activity jumps.

\subsection{The Borel tail and the L\'evy measure}

The Dwass total-progeny identity \cite[formula~(1)]{Dwass1969} gives, in the Poisson case, the Borel formula
\begin{equation}\label{eq:Borel-pmf}
    \mathbb P(S_T=k)
    =e^{-a_Tk}\frac{(a_Tk)^{k-1}}{k!},
    \qquad k\ge1.
\end{equation}
Stirling's formula gives, uniformly when $r_T^2k$ remains in compact subsets of $(0,\infty)$,
\begin{equation}\label{eq:Borel-asymptotic}
    \mathbb P(S_T=k)
    \sim \frac{1}{\sqrt{2\pi}}k^{-3/2}
      \exp\left(-\frac{r_T^2k}{2}\right).
\end{equation}
Indeed, the exponential rate is $a_T-1-\log a_T\sim r_T^2/2$.  The L\'evy measure of the collapsed compound Poisson process is
\[
    \nu_T
    :=T\mu_T\sum_{k=1}^\infty
      \mathbb P(S_T=k)\,\delta_{r_T^2k}.
\]
For every $f\in C_c((0,\infty))$, the uniform asymptotic \eqref{eq:Borel-asymptotic} and a Riemann-sum argument give
\[
    \int_0^\infty f(x)\nu_T(\dd x)
    \longrightarrow
    \frac{\vartheta}{\sqrt{2\pi}}
      \int_0^\infty f(x)x^{-3/2}e^{-x/2}\dd x.
\]
Thus, $\nu_T$ converges vaguely on $(0,\infty)$ to precisely the measure in \eqref{eq:IG-Levy}.  The singularity $x^{-3/2}$ has infinite mass but finite first moment near zero.  The limiting clock consequently has infinitely many jumps on every interval but finite variation.

Asymptotic \eqref{eq:Borel-asymptotic} fixes the constants and
provides a concrete microstructure interpretation: a limiting jump of size $x$ is the mass of a Hawkes family with approximately $x/r_T^2$ events.  Infinite activity results from the triangular balance between an increasing number of immigrants and the critical $k^{-3/2}$ family-size mass-function asymptotic.

\subsection{Brownian exploration and first passage}

There is an important distinction between one family and the immigrant forest $\mathcal F_T(t)$ defined in \eqref{eq:immigrant-forest}.
Explore a Poisson Galton--Watson family one individual at a time.  If $\xi_{T,i}\sim\Pois(a_T)$ are the offspring counts, its queue-length walk has increments $\xi_{T,i}-1$, mean $-r_T$, and variance $a_T\to1$.  The total progeny is the first time the walk, started from one, hits zero.  A single initial individual disappears under the space scale $r_T$, and indeed
\[
    \mathbb P\left(r_T^2S_T>\varepsilon\right)
    \le \frac{r_T}{\varepsilon}
    \longrightarrow0,
    \qquad \varepsilon>0.
\]
Thus a single unconditioned family does not converge to a non-degenerate IG variable.  Its order-$r_T^{-2}$ tail instead contributes the limiting excursion, or L\'evy, measure described above.

The non-degenerate first-passage picture appears at the forest level.  For fixed $t>0$, the root count $K_T(t)$ in \eqref{eq:immigrant-forest} satisfies
\[
    K_T(t)\sim\Pois(T\mu_Tt).
\]
Conditional on $K_T(t)$, concatenate the explorations of these families.  The resulting queue-length walk starts from $K_T(t)$ and hits zero after
$\sum_{j=1}^{K_T(t)}S_{T,j}$
steps.  The following proposition makes the exploration limit rigorous.

\begin{proposition} \label{prop:forest-passage}
For every fixed $t\ge0$, we have as $T\to\infty$, 
\begin{equation}\label{eq:forest-passage}
    r_T^2\sum_{j=1}^{K_T(t)}S_{T,j}
    \Longrightarrow
    \tau_{\vartheta t}
    :=\inf\{s\ge0:\vartheta t+B_s-s<0\},
\end{equation}
where $B$ is standard Brownian motion.  Moreover,
\begin{equation}\label{eq:forest-passage-LT}
    \E\left[e^{-u\tau_{\vartheta t}}\right]
    =\exp\left\{-\vartheta t
      \left(\sqrt{1+2u}-1\right)\right\}.
\end{equation}
The strict-crossing convention makes $x\mapsto\tau_x$ the
right-continuous inverse of $s\mapsto\sup_{0\le v\le s}(v-B_v)$.
For each fixed $x\ge0$, it agrees almost surely with the first hitting
time of level $x$.  As the initial level varies,
$(\tau_x)_{x\ge0}$ is the Brownian first-passage subordinator; hence
$(\tau_{\vartheta t})_{t\ge0}$ has the same c\`adl\`ag law as $Y$.
\end{proposition}

\begin{proof}
Let
\[
    Q_n^T:=K_T(t)+\sum_{i=1}^n(\xi_{T,i}-1).
\]
The total progeny of the forest is the first $n$ for which $Q_n^T=0$.  By the scaling assumptions of Corollary~\ref{cor:scalar-IG},
\[
    r_TK_T(t)\longrightarrow\vartheta t
    \quad\hbox{in probability},
    \qquad
    \operatorname{Var}\left(r_TK_T(t)\right)
    =\vartheta_Tr_Tt\longrightarrow0.
\]
Define the centered partial-sum process
\[
    B_T(s):=r_T\sum_{i=1}^{\lfloor s/r_T^2\rfloor}
       (\xi_{T,i}-a_T),
    \qquad s\ge0.
\]
Its predictable variance satisfies
\[
    \lfloor s/r_T^2\rfloor r_T^2a_T\longrightarrow s.
\]
Moreover, the fourth centered moments of $\Pois(a_T)$ are uniformly bounded because $0<a_T<1$.  Hence, for every $S,\varepsilon>0$,
\begin{align*}
 \sum_{i=1}^{\lfloor S/r_T^2\rfloor}
   \E\left[
      r_T^2(\xi_{T,i}-a_T)^2
      \1_{\{r_T|\xi_{T,i}-a_T|>\varepsilon\}}
   \right]\le
   \frac{r_T^4}{\varepsilon^2}
   \lfloor S/r_T^2\rfloor
   \E\left[(\xi_{T,1}-a_T)^4\right]
   \longrightarrow0.
\end{align*}
The martingale functional central limit theorem \cite[Theorem~VIII.3.11]{JacodShiryaev2003} therefore gives $B_T\Longrightarrow B$ locally in $J_1$.  Since
\[
    r_TQ^T_{\lfloor s/r_T^2\rfloor}
    =r_TK_T(t)+B_T(s)
      +r_T\lfloor s/r_T^2\rfloor(a_T-1),
\]
and
\[
    r_T\lfloor s/r_T^2\rfloor(a_T-1)\longrightarrow-s,
\]
locally uniformly in $s$, Slutsky's theorem yields
\[
    r_TQ^T_{\lfloor s/r_T^2\rfloor}
    \Longrightarrow \vartheta t+B_s-s,
\]
locally in $J_1$.

By \cite[Theorem~13.6.4]{Whitt2002}, the first-passage functional is continuous at continuous paths that cross strictly below the boundary arbitrarily soon after their finite first hit.  To verify this property, fix $x>0$ and write $\widehat\tau_x:=\inf\{s\ge0:x+B_s-s=0\}$, which is finite almost surely.  By the strong Markov property, $\widehat B_s:=B_{\widehat\tau_x+s}-B_{\widehat\tau_x}$ is standard Brownian motion.  The reflection principle then gives
$\inf_{0<s<\varepsilon}\widehat B_s<0$
for every $\varepsilon>0$, almost surely;
see \cite[Sections~2.6.A and~2.6.C]{KaratzasShreve1991}.  Consequently,
$x+B_{\widehat\tau_x+s}-(\widehat\tau_x+s)=\widehat B_s-s$ takes strictly negative values arbitrarily soon after $\widehat\tau_x$.
Thus $\tau_x=\widehat\tau_x$ almost surely for this fixed level.
The queue walk is downward skip-free, so its first nonpositive value
is zero.  The continuous-mapping theorem for this first-passage
functional therefore proves \eqref{eq:forest-passage} for $t>0$;
for $t=0$, Brownian regularity gives $\tau_0=0$, and the queue is empty.

Equation~\eqref{eq:scalar-collapsed-limit} also gives
$r_T^2\sum_{j=1}^{K_T(t)}S_{T,j}\Rightarrow Y_t$.
Uniqueness of the weak limit therefore identifies
$\tau_{\vartheta t}\overset{\mathrm d}=Y_t$, and
\eqref{eq:IG-LT} gives \eqref{eq:forest-passage-LT}.
The strong Markov property gives stationary independent increments
in the level $x$.  Together with the right-continuous inverse convention,
this identifies the c\`adl\`ag first-passage subordinator.
\end{proof}

Thus the two routes to IG are the same probabilistic mechanism viewed at different levels:
\[
  \text{near-critical forest}
  \longrightarrow
  \text{Brownian exploration with drift}
  \longrightarrow
  \text{first-passage time}
  \longrightarrow
  \text{IG mass}.
\]
Brownian first-passage limits for branching total mass and busy periods appear in \cite{PardouxWakolbinger2011,AbateWhitt1995}, while closely related inverse-Gaussian reproduction measures arise in rare-mutation limits of Galton--Watson trees \cite{Bertoin2009,Bertoin2010}.

\subsection{Finite-variance universality for age-dependent branching clusters}

Poisson offspring is built into the immigration--branching representation of a linear Hawkes process, but it is not essential for the limiting family-mass law.  The following formulation applies to age-dependent branching cluster processes, including marked or randomized reproduction mechanisms that need not possess a linear Hawkes intensity.

For every $T$, suppose the offspring counts of different individuals are independent and identically distributed with probability generating function $f_T$ and mean $a_T=f_T'(1)<1$, and set $r_T:=1-a_T$.  Child delays are independent with the fixed density $\phi$, independently of all offspring counts.  Immigrants form a Poisson process of rate $\mu_T$, independent of all reproduction mechanisms.  Let $\widetilde N^T$ denote the associated cluster count, started with empty history, and set $\widetilde X_t^T=r_T^2\widetilde N_{Tt}^T$.  Assume that $r_T\to 0$, and for some $\sigma^2\in(0,\infty)$, the uniform local expansion
\begin{equation}\label{eq:offspring-expansion}
 \log f_T(1-r_Tz)
 =-a_Tr_Tz
  +\frac{\sigma^2-1}{2}r_T^2z^2
  +o(r_T^2),
\end{equation}
locally uniformly for $z\ge0$.  More concretely, if $\xi_T$ has probability generating function $f_T$, $\operatorname{Var}(\xi_T)\to\sigma^2$, and $\sup_T\E[\xi_T^3]<\infty$, then Taylor expansion of $f_T(1-r_Tz)$ and of the logarithm gives \eqref{eq:offspring-expansion}, uniformly for $z$ in compact subsets of $[0,\infty)$.

The expansion \eqref{eq:offspring-expansion} replaces the unit quadratic coefficient of Poisson branching by the limiting offspring variance $\sigma^2$.  The next result identifies the resulting IG parameters and verifies calendar collapse through the same renewal first moment \eqref{eq:age-moment}.

\begin{proposition} \label{prop:universality}
For the age-dependent branching cluster process just defined, assume that conditions in Corollary~\ref{cor:scalar-IG} and   \eqref{eq:offspring-expansion} holds.  Then, for every $L>0$, as $T\to\infty$,
\[
    \widetilde X^T\Longrightarrow Y^{(\sigma)},
\]
in $\cD([0,L],\R)$ equipped with $M_1$-topology, where $Y^{(\sigma)}$ is the subordinator with Laplace exponent
\begin{equation}\label{eq:universal-exponent}
    \psi_\sigma(u)
    =\frac{\vartheta}{\sigma^2}
      \left(\sqrt{1+2\sigma^2u}-1\right),
\end{equation}
 and L\'evy density  
\begin{equation}\label{eq:universal-Levy}
    \nu_\sigma(\dd x)
    =\frac{\vartheta}{\sigma\sqrt{2\pi}}x^{-3/2}
       \exp\left(-\frac{x}{2\sigma^2}\right)
       \1_{(0,\infty)}(x)\dd x.
\end{equation}
In particular, the Poisson-Hawkes case corresponds to $\sigma^2=1$.  
As a consistency check,
\begin{equation}\label{eq:universal-mean-check}
    \psi_\sigma'(0)=\vartheta,
\end{equation}
independently of $\sigma^2$, in agreement with $\E[S_T]=r_T^{-1}$.
\end{proposition}

\begin{proof}
Let $S_T$ denote the total progeny of one family with offspring generating function $f_T$, and set
\[
    q_T(u):=\E\left[e^{-ur_T^2S_T}\right].
\]
By the branching property, $S_T\overset{\mathrm d}=1+\sum_{j=1}^{\xi_T}S_{T,j}$, where $\xi_T$ has generating function $f_T$ and the $S_{T,j}$ are independent copies of $S_T$, independent of $\xi_T$.  Conditioning on $\xi_T$ gives the total-progeny branching equation \cite[Section~2.1, p.~231]{Minami2005}:
\begin{equation}\label{eq:universal-branching-equation}
    q_T(u)=e^{-ur_T^2}f_T(q_T(u)).
\end{equation}
Write $q_T=1-r_Tz_T$.  The inequality $1-e^{-x}\le x$ and $\E[S_T]=r_T^{-1}$ give $z_T\le u$, so $(z_T)$ is bounded.  Taking logarithms and using \eqref{eq:offspring-expansion} together with
\[
    -\log(1-r_Tz)
    =r_Tz+\frac{r_T^2z^2}{2}+o(r_T^2),
\]
gives
\[
    u=z_T(u)+\frac{\sigma^2z_T(u)^2}{2}+o(1).
\]
The unique nonnegative root is
\[
    z=\frac{\sqrt{1+2\sigma^2u}-1}{\sigma^2}.
\]
Let $\overline{\widetilde X}^T$ denote the version of $\widetilde X^T$ in which every family is placed at its immigrant time.  Its Laplace exponent is
\[
    \widetilde\psi_T(u)
    =T\mu_T(1-q_T(u))
    =\vartheta_Tz_T(u)
    \longrightarrow\psi_\sigma(u).
\]
Pointwise convergence of these Laplace exponents, together with continuity of $\psi_\sigma$ at zero, gives $\overline{\widetilde X}^T\Longrightarrow Y^{(\sigma)}$ in $J_1$ by the L\'evy-process convergence criterion
\cite[Corollary~VII.3.6]{JacodShiryaev2003}.

The expected number of individuals in generation $n$ remains $a_T^n$.  Thus its expected-age measure is still
\eqref{eq:renewal-measure}.  Campbell's formula and the tail integral
calculation in the proof of Lemma~\ref{lem:multitype-collapse}
therefore give
\[
    0\le
    \E\left[\overline{\widetilde X}_t^T-\widetilde X_t^T\right]
    \le\mu_Tm_\phi a_T\to 0,
\]
and the mean event measure yields, for $0\le s\le t$,
\begin{equation}\label{eq:universal-increment-bound}
    \E\left[\widetilde X_t^T-\widetilde X_s^T\right]
    \le(t-s)\vartheta_T.
\end{equation}
Thus finite-dimensional convergence transfers from $\overline{\widetilde X}^T$ to $\widetilde X^T$.  Since $\widetilde X^T$ is nonnegative and nondecreasing, the increment bound \eqref{eq:universal-increment-bound} gives compact containment and both endpoint estimates by Markov's inequality, as in the proof of
Lemma~\ref{lem:multitype-collapse}.  The monotone-path tightness
criterion \cite[Theorem~12.12.3]{Whitt2002} therefore gives
$M_1$ tightness of $(\widetilde X^T)$.  This proves $\widetilde X^T\Longrightarrow Y^{(\sigma)}$ in $M_1$.

To verify the L\'evy density, use
\[
    \int_0^\infty(1-e^{-ux})x^{-3/2}e^{-\beta x}\dd x
    =2\sqrt{\pi}\left(\sqrt{\beta+u}-\sqrt{\beta}\right),
    \qquad \beta>0.
\]
Taking $\beta=(2\sigma^2)^{-1}$ gives
\[
    \int_0^\infty(1-e^{-ux})\nu_\sigma(\dd x)
    =\frac{\vartheta}{\sigma^2}
       \left(\sqrt{1+2\sigma^2u}-1\right)
    =\psi_\sigma(u),
\]
which proves \eqref{eq:universal-Levy}.  Differentiating \eqref{eq:universal-exponent} at zero gives \eqref{eq:universal-mean-check}.
\end{proof}

Proposition~\ref{prop:universality} identifies three features of the age-dependent branching limit: the $x^{-3/2}$ L\'evy singularity comes from finite-variance critical branching, the exponential tempering comes from the cutoff $r_T^{-2}$, and $M_1$ collapse comes from finite-mean calendar delays.  Heavy-tailed offspring or marks lead to different stable classes, as in \cite[Theorems~2.7 and~2.10]{Talarczyk2025} and \cite[Theorem~2.6]{Talarczyk2026}.

\section{Conclusion}\label{sec:conclusion}

We have shown that a nearly critical multivariate Hawkes system can converge directly to a multitype inverse-Gaussian clock.  The first-order matrix defect
$\boldsymbol B$ records cross-type interaction, and the coupled quadratic system
\eqref{eq:multitype-IG-root} determines the joint Laplace exponent.  The same process is the first-passage subordinator of a Brownian additive L\'evy field.  Thus the multivariate limit is not introduced by an arbitrary dependence construction: its dependence is inherited from microscopic cross-excitation.

The cluster proof isolates the mechanism.  Multitype finite-variance branching produces rare families with masses of order $r_T^{-2}$; the balance
$Tr_T\boldsymbol\mu_T\to\boldsymbol\eta$ keeps a non-degenerate collection of such families visible; and the matrix age estimate
\eqref{eq:matrix-age-moment}, together with $Tr_T\to\infty$, collapses their internal calendar time.  Product $M_1$ merges the resulting microscopic staircases into macroscopic jumps.  When $\boldsymbol B$ is diagonal, the coupled exponent separates into independent classical IG clocks; $d=1$ gives the scalar result, and the square-root normalization reproduces the hyper-rough-boundary law of \cite[Theorem~3.1]{AbiJaberAttalRosenbaum2025}.

Our comparison with \cite[equations~(2.8)--(2.9), Theorem~2.9, and Remark~2.21]{Xu2024Scaling} draws a sharp boundary.  Its scaling theory depends on accumulated activity $n\theta_n$, so the assumption
$\theta_n\to\infty$ can be weakened to $n\theta_n\to\infty$ under common tilted stability and the diagonal-atom condition.  The positive-resolvent and local Riccati arguments above establish this limit, and monotonicity upgrades the count and compensator to product $M_1$.  For our array the formal potential is $\boldsymbol B^{-1}\delta_0$.  The diagonal-atom condition \cite[Assumption~2.8 and Remark~2.21]{Xu2024Scaling} therefore proves the diagonal multitype and scalar cases, while the non-diagonal theorem requires the direct cluster identification and $M$-matrix uniqueness argument given here.

In terms of financial interpretation, a jump of $Z_i$ represents a rapidly completed endogenous avalanche of type-$i$ activity, while dependence among coordinates records how one event class propagates into others.  Corollary~\ref{cor:NIG} shows that a Brownian motion run on the scalar square-root clock is symmetric NIG, but a process-level return model would additionally need signs, imbalance, leverage, and a rule for composition at clock jumps \cite{ContTankov2004,AbiJaberDeCarvalho2024}.  Heavy-tailed delays may preserve family timing and produce rough memory, whereas infinite-variance offspring or marks should lead to stable rather than IG clocks.  Extending the genuinely coupled result beyond the identity critical eigenspace, and treating those tail regimes jointly with multitype interaction, remain natural directions.

\bibliographystyle{plain}
\bibliography{bibtex}

\end{document}